%% file: main.tex
\documentclass[12pt, reqno]{amsart}
\usepackage{xcolor}
\input{preamble}

\title{Rank stability of elliptic curves over $\mathbb{F}_q(t)$ in residue classes}
\author{Steve Fan}
\address{Steve Fan, Department of Mathematics\\ University of Georgia\\ Athens, GA 30602}
\email{Steve.Fan@uga.edu}
\author{Sun Woo Park}
\address{Sun Woo Park, Max Planck Institute for Mathematics,  Vivatsgasse 7, 53111 Bonn, Germany}
\email{s.park@mpim-bonn.mpg.de}
\begin{document}

\begin{abstract}
Fix a prime $\ell$. Let $K = \mathbb{F}_q(t)$ be a global function field such that $\gcd(q,6) = 1$ and $q \equiv 1 \pmod \ell$. Let $E$ be a non-isotrivial elliptic curve over $K$. Given a fixed monic polynomial $Q$ over $\mathbb{F}_q$, and assuming some mild conditions on $E$, we show that the rank of $E$ does not change with respect to a positive proportion of $\mathbb{Z}/\ell \mathbb{Z}$ extensions $K(\sqrt[\ell]{f})/K$, as $f$ varies over the set of monic polynomials over $\mathbb{F}_q$ such that $f \equiv A \pmod Q$ for any given polynomial $A$. We obtain this result by combining analytic and probabilistic techniques to study the distribution of certain prime Selmer groups of auxiliary abelian varieties constructed from these polynomials $f$.
\end{abstract}
\maketitle
% \tableofcontents

\section{Introduction}

Let $E$ be an elliptic curve over a global field $K$. By the Mordell--Weil theorem, its set of $K$-rational points, denoted by $E(K)$, is a finitely generated abelian group. As a consequence, we have an isomorphism $E(K) \cong \mathbb{Z}^{\oplus r} \oplus T$, where $T$ is a finite abelian group. We call $r$ the rank of $E(K)$ and write $r=\mathrm{rank}(E/K)$.

Given a field extension $L/K$, it is natural to ask how the rank of the elliptic curve $E$ varies with respect to base change from $K$ to $L$. In particular, one can ask for the distribution of the quantity $\mathrm{rank}(E/L) - \mathrm{rank}(E/K)$ as $L$ varies over a family $\mathcal{F}$ of some field extensions of $K$. Even the case when $\mathcal{F}$ is a family of quadratic extensions over $K$ seems highly non-trivial. Goldfeld's conjecture (or the minimalist conjecture) predicts the following behavior of differences in ranks:
\[
    \lim_{X \to \infty} \mathbb{P}[\mathrm{rank}(E/L) - \mathrm{rank}(E/K) = j : \mathrm{Gal}(L/K) \cong \mathbb{Z}/2\mathbb{Z}, \mathrm{Disc}(L) \leq X] = \begin{cases}
        \frac{1}{2} &\text{ if } j = 0, 1, \\
        0 &\text{ otherwise}.
    \end{cases}
\]
Significant progress has been made in supporting this conjecture. The state-of-the-art works include recent series of published papers by Alexander Smith \cite{Sm22_01, Sm22_02} when $K$ is a number field and those by Jordan Ellenberg, Aaron Landesman, and Ishan Levy \cite{EL23, LL25} when $K$ is a global function field. Together, they show that in both settings the Birch and Swinnerton--Dyer conjecture implies Goldfeld's conjecture, modulo some mild technical conditions on the elliptic curve $E$.

The statistics of rank growths of elliptic curves over \textit{proper subsets} of field extensions over $K$ also has many interesting applications in number theory. Research on this topic has mostly focused on the case where $K$ is a number field. One such example concerns the collection of field extensions of form $\mathbb{Q}(\sqrt{f(x_1, x_2 ,\cdots, x_n)})$, where $f$ is some multivariate polynomial over $\mathbb{Z}$, and each $x_k$ vary over some interval $I_k \subseteq \mathbb{Q}$. The recent breakthrough made by Peter Koymans and Carlo Pagano \cite{KP27} shows that Hilbert's 10th problem over number rings has a negative answer, by constructing elliptic curves $E: f(a)y^2 = f(x)$ over any number field $K$ and some $a \in K$ such that its algebraic rank over $K(\sqrt{-1})$ and $K$ are identical and positive. Another relevant work of theirs, jointly with Efthymios Sofos \cite{KPS24}, concerns boundedness of exponential moments of ranks of specializations of elliptic fibrations over $\mathbb{A}^n_{\mathbb{Q}}$. In addition, recent work by Peter Koymans and Alex Smith \cite{KS26} provides criteria for boundedness of moments of Selmer groups of quadratic twists of abelian varieties using Tamagawa ratios. Finally, the work by Alex Bartel and Adam Morgan \cite{BM25} considers rank growths of Jacobians of hyperelliptic curves over proper subsets of quadratic extensions over $\mathbb{Q}$ which satisfy a predetermined set of Chebotarev conditions.

In the spirit of these previous studies, we study rank growths of elliptic curves over a certain subset of cyclic prime Galois extensions over a global function field satisfying some congruence relations. Fix a prime number $\ell$ and a global function field $K = \mathbb{F}_q(t)$ with a characteristic coprime to $6$, satisfying $q \equiv 1 \pmod \ell$. Let $Q$ be a polynomial over $\mathbb{F}_q$, and let $E$ be a non-isotrivial elliptic curve over $K$, i.e., its $j$-invariant is a non-constant rational function in $t$. Our goal is to understand how the rank of $E$ varies with respect to base change over \textit{subsets} of $\mathbb{Z}/\ell \mathbb{Z}$ Galois extensions of the form $K(\sqrt[\ell]{f})/K$ as $f$ varies within a residue class modulo $Q$. We prove the following rank stability result under some mild assumptions on $E$.
\begin{theorem} \label{theorem:main1}
    Let $\ell$ be a prime number. Let $K = \mathbb{F}_q(t)$ be a global function field of characteristic coprime to $6$ such that $q \equiv 1 \pmod \ell$. Suppose $E/K$ is an elliptic curve satisfying the following three conditions.
    \begin{itemize}
        \item $E/K$ is a non-isotrivial elliptic curve.
        \item $\mathrm{Gal}(K(E[\ell])/K) \cong \mathrm{SL}_2(\mathbb{F}_\ell)$.
        \item $E/K$ has a place of split multiplicative reduction.
    \end{itemize}
    Let $Q$ be a fixed monic polynomial over $\mathbb{F}_q$, and let $A$ be any polynomial over $\mathbb{F}_q$. Denote by $\mathcal{M}(n;Q,A)$ the set of monic polynomials $f$ of degree $n$ over $\mathbb{F}_q$ such that $f \equiv A \pmod Q$. Then 
    \[
        \liminf_{n \to \infty} \frac{\#\{f \in \mathcal{M}(n;Q,A) : \mathrm{rank}(E/K(\sqrt[\ell]{f})) = \mathrm{rank}(E/K) \}}{\# \mathcal{M}(n;Q,A)} \geq \prod_{i=0}^\infty \frac{1}{1+\ell^{-i}}.
    \]
\end{theorem}
This theorem generalizes the case $Q = 1$ due to Park \cite{Park2025}. Our result in particular shows that regardless of the imposed congruence relation, a positive proportion of such subsets of $\mathbb{Z}/\ell \mathbb{Z}$ extensions $L$ over $K$ satisfy $\mathrm{rank}(E/L) = \mathrm{rank}(E/K)$.

\subsection*{Strategy and Organization}
The proof rests upon computing the distribution of prime Selmer groups of auxiliary abelian varieties, whose rank of $K$-rational points detect the difference between $\mathrm{rank}(E/K(\sqrt[\ell]{f}))$ and $\mathrm{rank}(E/K)$. We start by constructing an auxiliary abelian variety for each $f \in \mathcal{M}(n; Q, A)$ by following the work by Mazur, Rubin, and Silverberg \cite{MR07, MRS07}:
\begin{equation*}
    E^{\chi_f} := \mathrm{Ker}(\mathrm{Res}_K^{K(\sqrt[\ell]{f})} E \to E).
\end{equation*}
Here, $\mathrm{Res}_K^{K(\sqrt[\ell]{f})} E$ is the Weil restriction of scalars of $E$ with respect to field extension $K(\sqrt[\ell]{f})/K$. There is a natural permutation action of the generator $\sigma_f \in \mathrm{Gal}(K(\sqrt[\ell]{f})/K)$ on $E^{\chi_f}$. Using this generator, we can define an auxiliary $\mathbb{F}_\ell$-vector space, called the $1-\sigma_f$ Selmer group of $E^{\chi_f}$ and denoted by $\mathrm{Sel}_{1-\sigma_f}(E^{\chi_f}/K)$, and obtain
\begin{equation*}
    \mathrm{rank}(E/K(\sqrt[\ell]{f})) - \mathrm{rank}(E/K) \leq (\ell - 1) \dim_{\mathbb{F}_\ell} \mathrm{Sel}_{1-\sigma_f}(E^{\chi_f}/K).
\end{equation*}
In the aforementioned work by Park \cite{Park2025} (and the work \cite{KMR14} by Klagsbrun, Mazur, and Rubin for the number field case), the distribution of such Selmer groups as $f$ varies over the full set of monic degree $n$ polynomials over $\mathbb{F}_q$ can be obtained using a combination of analytic and probabilistic techniques, including effective Chebotarev density theorems, equidistribution of restriction of global characters to Cartesian products of local characters, and stochastic properties of associated Markov operators. However, these results do not automatically transfer when $f$ is restricted to a subset of polynomials subject to a given congruence relation. To overcome this problem, we require additional non-trivial ingredients, including a function-field analogue of Tur\'an's theorem and a new framework for anatomy of polynomials satisfying any given congruence relation. We also need to ensure that the distribution of these Selmer groups can still be obtained after these new ingredients are incorporated. We demonstrate how these two goals can be achieved, thereby proving that the same lower bound obtained in \cite{Park2025} also holds for rank stability of elliptic curves with respect to base change over subsets of $\mathbb{Z}/\ell \mathbb{Z}$ extensions subject to a fixed congruence relation.

The paper is organized as follows. In Section \ref{sec:preliminary}, we discuss the function field analogue of Tur\'an's theorem mentioned above, as well as the anatomy of polynomials satisfying congruence relations. Some additional analytic inputs such as effective versions of Chebotarev density theorem are also recalled. In Section \ref{sec:changes_rank}, we prove our main theorem by constructing auxiliary abelian varieties for each polynomial $f \in \mathcal{M}(n;Q, A)$ using Weil restrictions, introduce Selmer groups of such abelian varieties, gather stochastic properties of associated Markov operators, and compute the distribution of such Selmer groups.

\subsection*{Acknowledgements}
We would like to thank the Max Planck Institute for Mathematics for its hospitality and support, where this project started. We would also like to thank Peter Koymans for sharing his joint works with Carlo Pagano and Efthymios Sofos \cite{KPS24} and with Alexander Smith \cite{KS26}, from which we took inspiration for our project on global function fields.
% interest on this problem over global function fields started.

\section{Preliminaries} \label{sec:preliminary}

This section introduces relevant analytic inputs and results on anatomy of polynomials over finite fields. We will use these ingredients to understand changes in algebraic ranks of elliptic curves with respect to cyclic prime extensions over $K$.

\subsection{Analytic inputs}

Fix a prime power $q$. For any polynomial $f\in\mathbb{F}_q[t]$, define $|f|:=q^{\deg f}$. We denote by $\M(n)$ the set of monic polynomials of degree $n$ over $\mathbb{F}_q$. We denote by $\mathcal{P}$ the set of monic irreducible polynomials over $\mathbb{F}_q$, and by $\mathcal{P}(n)$ the subset of $\mathcal{P}$ consisting of monic irreducible polynomials of degree $n$. Writing $\pi(n):=\#\PP(n)$, we have by the Prime Number Theorem \cite[Theorem 2.2]{Rosen02} that
\begin{equation}\label{eq:PNT}
\pi(n)=\frac{q^n}{n}+O_q\left(\frac{q^{n/2}}{n}\right)
\end{equation}
for all $n\ge1$. By induction, one sees that
\begin{align}\label{eq:HR}
\pi_k(n):&=\#\{f\in\M(n): f=g_1\cdots g_k~\text{with}~g_1,...,g_k\in\PP~\text{distinct}\}\nonumber\\
&=\frac{1}{(k-1)!}\cdot\frac{q^n}{n}(\log(2n))^{k-1}+O_{k,q}\left(\frac{q^n}{n}((\log(2n))^{k-2}\right)  
\end{align}
for all $n\ge1$ and every fixed $k\in\N$. A more precise asymptotic formula for $\pi_k(n)$ which holds uniformly in the range $1\le k\ll \log(2n)$ has been established by Afshar and Porritt \cite{AP19}, but the special case $k=2$ of \eqref{eq:HR} will be sufficient for our applications.

Given $f \in \M(n)$ and $m\in\Z_{\ge0}$, we denote by $\omega(f)$ the number of distinct monic irreducible factors of $f$ and by $\omega_{\le m}(f)$ the number of distinct monic irreducible factors of $f$ of degree at most $m$. For any $Q\in\M(m)$ and $A\in\mathbb{F}_q[t]$, recall that
\begin{equation*}
    \M(n;Q,A) := \{ f \in \M(n) : f \equiv A \pmod Q \}.
\end{equation*}
Evidently, $\#\M(n;Q,A)=q^{n-m}$ when $n\ge m$ and $\#\M(n;Q,A)\le1$ when $n<m$. Likewise, we denote by $\mathcal{P}(n;Q,A)$ the subset of $\mathcal{P}(n)$ defined as
\begin{equation*}
    \mathcal{P}(n;Q,A) := \{g \in \mathcal{P}(n) : g \equiv A \pmod Q\}.
\end{equation*}

% We write $\pi(n):=\#\PP(n)$ and $\pi(n;Q,A):=\#\mathcal{P}(n;Q,A)$ and $\pi(n):=\pi(n;1,1)$. The Prime Number Theorem for arithmetic progressions asserts that
% \begin{equation}\label{eq:PNTAP}
% \pi(n;Q,A)=\frac{1}{\varphi(Q)}\cdot\frac{q^n}{n}+O\left(\frac{q^{n/2}m}{n}\right)
% \end{equation}
% uniformly for all $Q\in\M(m)$ and $A\in\mathbb{F}_q[t]$ such that $\gcd(Q,A)=1$, where 
% \[\varphi(Q):=q^{m}\prod_{\substack{g\in\PP\\g\mid Q}}\left(1-\frac{1}{q^{\deg g}}\right)\]
% is the Euler totient function. This estimate follows from the Generalized Riemann Hypothesis for $\mathbb{F}_q[t]$ proved by A. Weil in 1948.

Our first analytic input is the following function-field analogue of Tur\'{a}n's theorem for $\omega$ over $\M(n;Q,A)$. The proof is straightforward and is included here for the sake of completeness.
\begin{lemma}[Tur\'{a}n's theorem for $\omega$ on $\M(n;Q,A)$]\label{Lem:TuranAP}
Let $q$ be a prime power and let $n\ge m\ge1$ be integers. Then  
\begin{equation}\label{eq:varomega}
\frac{1}{q^{n-m}}\sum_{f\in\M(n;Q,A)}\left(\omega(f)-\log n\right)^2\ll_{m,q}\log(2n)   
\end{equation}
for all $Q\in\M(m)$ and $A\in\mathbb{F}_q[t]$. Consequently, we have
\[\frac{\#\{f\in\M(n;Q,A):|\omega(f)-\log n|\ge B\log n\}}{q^{n-m}}\ll_{m,q}\frac{1}{B^2\log(2n)}\]
for any $B>0$. 
\end{lemma}
\begin{proof}
It suffices to prove \eqref{eq:varomega}. Without loss of generality, we may assume that $n$ is sufficiently large in terms of $m$ and that $\deg A<m$. We begin by computing the mean value of $\omega$ on $\M(n;Q,A)$ defined by 
\[\E[\omega]:=\frac{1}{q^{n-m}}\sum_{f\in\M(n;Q,A)}\omega(f).\]
Let $D=\gcd(Q,A)$. By \eqref{eq:PNT} we have
\begin{align}\label{eq:1stomega}
\E[\omega]=\frac{1}{q^{n-m}}\sum_{\substack{g\in\PP\\\deg g\le n}}\sum_{\substack{f\in\M(n;Q,A)\\g\mid f}}1&=\frac{1}{q^{n-m}}\sum_{\substack{g\in\PP\\\deg g\le n\\ g\nmid D}}\sum_{\substack{f\in\M(n;Q,A)\\g\mid f}}1+O\left(\omega(D)\right)\nonumber\\
&=\frac{1}{q^{n-m}}\sum_{\substack{g\in\PP\\\deg g\le n-m}}q^{n-m-\deg g}+ O\left(\omega(D)+\frac{q^m}{n}\right)\nonumber\\
&=\frac{1}{q^{n-m}}\sum_{k=1}^{n-m}q^{n-m-k}\left(\frac{q^k}{k}+O\left(\frac{q^{k/2}}{k}\right)\right)+O_{m,q}(1)\nonumber\\
&=\log n + O_{m,q}(1).
\end{align}
Next, we compute the second moment of $\omega$ on $\M(n;Q,A)$ defined by 
\[\E[\omega^2]:=\frac{1}{q^{n-m}}\sum_{f\in\M(n;Q,A)}\omega(f)^2.\]
We start by observing that
\[\E[\omega^2]=\frac{1}{q^{n-m}}\sum_{f\in\M(n;Q,A)}\omega(f)+\frac{1}{q^{n-m}}\sum_{\substack{g_1\ne g_2\in\PP\\\deg(g_1g_2)\le n}}\sum_{\substack{f\in\M(n;Q,A)\\g_1g_2\mid f}}1.\]
The double sum above can be estimated analogously. Note that the contribution from those $g_1,g_2$ at least one of which divides $Q$ is 
\begin{align*}
&=\frac{2}{q^{n-m}}\sum_{\substack{g_1\ne g_2\in\PP\\\deg(g_1g_2)\le n\\g_1\mid D,g_2\nmid D}}\sum_{\substack{f\in\M(n;Q,A)\\g_1g_2\mid f}}1+ \frac{1}{q^{n-m}}\sum_{\substack{g_1\ne g_2\in\PP\\\deg(g_1g_2)\le n\\g_1, g_2\mid D}}\sum_{\substack{f\in\M(n;Q,A)\\g_1g_2\mid f}}1\\
&\le\frac{2}{q^{n-m}}\sum_{\substack{g_1\ne g_2\in\PP\\\deg(g_1g_2)\le n\\\deg g_2\le n-m\\g_1\mid D,g_2\nmid D}}q^{n-m-\deg g_2}+\frac{2}{q^{n-m}}\sum_{\substack{g_1\ne g_2\in\PP\\\deg(g_1g_2)\le n\\\deg g_2>n-m\\g_1\mid D,g_2\nmid D}}1+\omega(D)^2\\
&\le2\omega(D)\sum_{\substack{g_2\in\PP\\\deg g_2\le n-m}}q^{-\deg g_2}+\frac{2\omega(D)}{q^{n-m}}\sum_{\substack{g_2\in\PP\\\deg g_2\le n}}1+\omega(D)^2\\
&\ll_{m,q}\log n.
\end{align*}
It follows that
\[\frac{1}{q^{n-m}}\sum_{\substack{g_1\ne g_2\in\PP\\\deg(g_1g_2)\le n}}\sum_{\substack{f\in\M(n;Q,A)\\g_1g_2\mid f}}1=\frac{1}{q^{n-m}}\sum_{\substack{g_1\ne g_2\in\PP\\\deg(g_1g_2)\le n\\g_1,g_2\nmid Q}}\sum_{\substack{f\in\M(n;Q,A)\\g_1g_2\mid f}}1+O_{m,q}(\log n).\]
But \eqref{eq:HR} implies
\begin{align*}
\frac{1}{q^{n-m}}\sum_{\substack{g_1\ne g_2\in\PP\\\deg(g_1g_2)\le n\\g_1,g_2\nmid Q}}\sum_{\substack{f\in\M(n;Q,A)\\g_1g_2\mid f}}1&=\frac{1}{q^{n-m}}\sum_{\substack{g_1\ne g_2\in\PP\\\deg(g_1g_2)\le n-m}}q^{n-m-\deg(g_1g_2)}+O\left(\frac{1}{q^{n-m}}\sum_{\substack{g_1\ne g_2\in\PP\\n-m<\deg(g_1g_2)\le n}}1\right)\\
&=2\sum_{k=2}^{n-m}\frac{\pi_2(k)}{q^k}+O\left(\frac{1}{q^{n-m}}\sum_{k=2}^{n}\pi_2(k)\right)\\
&=2\sum_{k=2}^{n-m}\frac{1}{k}\left(\log k+O(1)\right)+O\left(\frac{q^m\log n}{n}\right)\\
&=(\log n)^2+O_{m,q}(\log n).
\end{align*}
Collecting the estimates above, we obtain
\[\E[\omega^2]=\frac{1}{q^{n-m}}\sum_{f\in\M(n;Q,A)}\omega(f)+(\log n)^2+O_{m,q}(\log n)=(\log n)^2+O_{m,q}(\log n).\]
Together with \eqref{eq:1stomega} this yields
\[\frac{1}{q^{n-m}}\sum_{f\in\M(n;Q,A)}\left(\omega(f)-\E[\omega]\right)^2=\E[\omega^2]-\E[\omega]^2\ll_{m,q}\log n.\]
Finally, \eqref{eq:varomega} follows from this estimate in conjunction with \eqref{eq:1stomega} and the inequality $(a+b)^2\le 2(a^2+b^2)$ for any $a,b\in\R$.
\end{proof}
%\textbf{Goal}: It would be great to prove the following result: Show the analogue of Turan's theorem or Hardy-Ramanujan theorem for the set $\M(n;Q,A)$. In particular, it would be great to show that a density $1$ subset of $\M(n;Q,A)$ (even with explicit upper bound on the complement of such a density $1$ subset would be great, but not necessary) satisfies
%    \begin{equation*}
%        C_1 \log n \leq \omega(f) \leq C_2 \log n
%    \end{equation*}
%for any constant $C_1 \in (0,1)$ and $C_2 \in (1,\infty)$. The following reference by Afshar and Porritt \cite{AP19} (which seems to prove the Sathe-Selberg analogue for function fields in arithmetic progressions) could be useful for us for square-free polynomials. It is also fine to restrict our set $\M(n)$ to be the set of monic square-free polynomials of degree $n$ over $\mathbb{F}_q$, if it makes the exposition easier.

The next result \cite[Proposition 4.13]{Park2025} ensures that most of the irreducible factors of $f$ belong to $\cup_{k\ge\mathfrak{n}}\M(k;Q,A)$ with $\mathfrak{n} := \lfloor 4 (\log n)^2 \rfloor$. 
% (One can take it to be any fixed power of $\log n$).
\begin{theorem}[{\cite[Theorem 2.11]{Park2026}} and {\cite[Proposition 4.13]{Park2025}}] \label{thm:prop4.13_park}
    Let $\mathfrak{n} := \lfloor 4 (\log n)^2 \rfloor$. Then 
    \begin{equation*}
        \#\Biggl\{f \in \M(n;Q,A) : \omega_{\le \mathfrak{n}}(f) < \left\lfloor \frac{\log n}{\log \log \log n} \right\rfloor \Biggr\} \ll_{Q} \frac{q^n}{n^{2.9}}
    \end{equation*}
for every $n > \exp(\exp(\exp(e)))$.
\end{theorem}
\begin{proof}
 This follows from \cite[Proposition 4.13]{Park2025} with $n + \deg Q$ in place of $n$.
\end{proof}

The final analytic tool needed is the following effective Chebotarev density theorem.
\begin{theorem}[c.f. {\cite[Corollary 3.2]{Park2025}}] \label{thm:effective_chebotarev}
    Let $L/K$ be a Galois extension such that the constant field of $L$ is $\mathbb{F}_q$. Let $G := \mathrm{Gal}(L/K)$, and let $S, S' \subseteq G$ be unions of conjugacy classes of $G$. Let $g_L$ be the genus of the field $L$. Suppose $q$ satisfies
    \begin{equation*}
        q^{n/2} - q^{n/4} > 2(\#G + g_L).
    \end{equation*}
    If $n \geq 2(\log 8 + \log(\# G + g_L))$, then 
    \begin{equation*}
        \left| \frac{\#\{g \in \mathcal{P}(n) : \mathrm{Frob}_g \in S\}}{\#\{g \in \mathcal{P}(n) : \mathrm{Frob}_g \in S'\}} - \frac{\# S}{\# S'} \right| < 16 \frac{\# S}{\# S'} (\#G + g_L) q^{-n/2}.
    \end{equation*}
\end{theorem}

\subsection{Conditions on elliptic curves}

We impose the following conditions on the elliptic curve $E$ over $K = \mathbb{F}_q(t)$. These conditions and definitions appear in \cite[Section 4.1]{Park2025}, but slight modifications are needed here.
\begin{condition}[c.f. {\cite[Condition 4.1]{Park2025}}] \label{condition}
  Fix a prime $\ell$ coprime to $6q$ and assume the following conditions hold for the elliptic curve $E/K$.
    \begin{itemize}
        \item $E$ is non-isotrivial, i.e. its $j$-invariant is a non-constant rational function in $\mathbb{F}_q(t)$.
        \item $E$ has a place of split multiplicative reduction.
        \item $(q,6) = 1$ and $q \equiv 1 \pmod \ell$. In particular, the set $\mu_\ell$ of primitive $\ell$-th roots of unity is contained in $K$.
        \item $\mathrm{Gal}(K(E[\ell])/K) \cong \mathrm{SL}_2(\mathbb{F}_\ell)$.
    \end{itemize}
\end{condition}

\begin{definition}[c.f. {\cite[Definition 4.2]{Park2025}}]
    Let $Q\in\M(m)$ and fix a prime $\ell$ and an elliptic curve $E$ satisfying Condition \ref{condition}. We adopt the following notation.
    \begin{itemize}
        \item $\Delta_E$: the discriminant of $E$.
        \item $\Sigma_E$: the set of elements in $\mathcal{P}$ dividing $\Delta_E$.
        \item $\sigma$: a square-free product of all elements in some finite subset of $\mathcal{P} \setminus \Sigma_E$.
        \item $\Sigma_E(\sigma)$: the set of elements that divide $\sigma$ or are in $\Sigma_E$.
        \item $d_{\Sigma_E(\sigma)}$: the sum of degrees of elements in $\Sigma_E(\sigma)$. 
        \item $K_g$: the local field of $K$ at $g$ for each $g \in \mathcal{P}$.
        \item For each $0 \leq k \leq 2$, we denote by $\mathcal{P}_k(n)$ the set
        \begin{equation*}
            \mathcal{P}_k(n) := \{g \in \mathcal{P}(n) \setminus \Sigma_E : \dim_{\mathbb{F}_\ell} E[\ell](K_g) = k\}.
        \end{equation*}
        \item For each polynomial $R$ of degree less than $\deg Q$, and any $0 \leq k \leq 2$, we define the set
        \begin{equation*}
            \mathcal{P}_k(n;Q, R) := \{g \in \mathcal{P}(n; Q,R) \setminus \Sigma_E : \dim_{\mathbb{F}_\ell} E[\ell](K_g) = k\}.
        \end{equation*}
    \end{itemize}
\end{definition}

\begin{remark}
    Let $L$ be the compositum of the ray class field $K_Q/K$ with modulus $Q + \infty$ (considered as a divisor in $K$, see \cite[Exercise 9.18, p. 147]{Rosen02}) and $K(E[\ell])$. By \cite[Proposition 4.3, Theorem 9.25]{Rosen02}, the constant field of $L$ is equal to the constant field of $K$. Hence, we can compute the density $\#\mathcal{P}_k(n; Q, R)/\#\mathcal{P}(n;Q, R)$ using Theorem \ref{thm:effective_chebotarev}. For example, if $(Q, \Delta_E) = 1$ or $\ell \geq 5$, then there exists an explicit constant $B_{Q,E}$ depending on $Q$ and $E$ such that
    \begin{equation*}
        \max \left\{\left| \frac{\# \mathcal{P}_0(n; Q, R)}{\# \mathcal{P}(n;Q, R)} - \left(1 - \frac{\ell}{\ell^2-1} \right) \right|, \left| \frac{\# \mathcal{P}_1(n; Q, R)}{\# \mathcal{P}(n;Q, R)} - \frac{1}{\ell} \right|, \left| \frac{\# \mathcal{P}_2(n; Q, R)}{\# \mathcal{P}(n;Q, R)} - \frac{1}{\ell^3 - \ell} \right|\right\} < B_{Q,E} q^{-n/2}.
    \end{equation*}
\end{remark}

\subsection{Splitting partitions}

To obtain the desired results on rank growths of elliptic curves with respect to subsets of cyclic prime extensions over $K$ subject to congruence conditions, we make a slight modification to the definitions and notation appearing in \cite[Sections 4.2, 5.2]{Park2025} regarding the anatomy of monic polynomials of degree $n$.

\begin{definition}[c.f. {\cite[Definition 4.5]{Park2025}}]
    Let $Q\in\M(m)$. Fix a prime $\ell$ and an elliptic curve $E$ satisfying Condition \ref{condition}. Given a positive integer $n$, put
    % denote by $\mathcal{I}(n,Q)$ the set of 4-tuple of indices
    \begin{equation*}
        \mathcal{I}(n,Q) := \Biggl\{(i,j,k,R) \in \mathbb{Z}_{\geq 0}^{\oplus 3} \oplus \mathbb{F}_q[t] : {1 \leq i \leq n, 1 \leq j \leq n, 0 \leq k \leq 2, \deg R < \deg Q} \Biggr\}
    \end{equation*}
    % Given a positive integer $n$, we denote by $\lambda_{n,Q}$ a finite set of $5$-tuples of elements defined as
    and
    \begin{equation*}
        \lambda_{n,Q} := \{(\lambda_{i,j,k,R}, i,j,k,R) : \lambda_{i,j,k,R} \geq 0, (i,j,k,R) \in \mathcal{I}(n,Q)\}.
    \end{equation*}
\end{definition}

\begin{definition}[c.f. {\cite[Definition 4.7]{Park2025}}]
    Given positive integers $n, w$, we say that $\lambda_{n,Q}$ is a splitting partition with respect to $(n,w)$ if it satisfies the following two conditions.
    \begin{itemize}
        \item $\displaystyle{\sum_{i,j=1}^n \sum_{k=0}^2 \sum_{\substack{R \in \mathbb{F}_q[t] \\ \deg R < \deg Q}} \lambda_{i,j,k,R} \cdot i \cdot j = n}$,
        \item $\displaystyle{\sum_{i,j=1}^n \sum_{k=0}^2 \sum_{\substack{R \in \mathbb{F}_q[t] \\ \deg R < \deg Q}} \lambda_{i,j,k,R} = w}$.
    \end{itemize}

    Suppose a monic polynomial $f \in \mathcal{M}(n;Q,A)$ admits an irreducible factorization
    \begin{equation*}
        f = \prod_{m=1}^w g_m^{j_m},
    \end{equation*}
    where $g_m \in \mathcal{P}$ and $j_m > 0$ for every $1 \leq m \leq w$. Then such an $f \in \mathcal{M}(n;Q,A)$ admits a splitting partition $\lambda_{n,Q}$ with respect to $(n,w)$ if it satisfies the following equation for all indices $(i,j,k,R) \in \mathcal{I}(n,Q)$:
    \begin{equation*}
        \lambda_{i,j,k,R} = \#\{g \in \mathcal{P}_k(i ; Q, R) : g^j \mid f \text{ but } g^{j+1} \nmid f\}.
    \end{equation*}
    In particular, for each index $(i,j,k,R) \in \mathcal{I}(n,Q)$, $\lambda_{i,j,k,R}$ is the number of distinct irreducible factors of $f$ of degree $i$ and multiplicity $j$ such that $\dim_{\mathbb{F}_\ell} E[\ell](K_g) = k$ and $g \equiv R \pmod Q$.
\end{definition}

\begin{definition} [c.f. {\cite[Definition 4.10]{Park2025}}] 
    Given a monic polynomial $f \in \mathcal{M}(n; Q, A)$ and an irreducible polynomial $g \in \mathcal{P}$, we denote by $v_g(f)$ the multiplicity of $g$ dividing $f$. Then we denote by $f_*$ and $f^*$ the polynomials defined as
    \begin{equation*}
        f_* := \prod_{i=1}^{\mathfrak{n}}\prod_{\substack{g \in \mathcal{P}(i) \\ g \mid f}} g^{v_g(f)}, \hspace{15pt} f^* := \prod_{i=\mathfrak{n}+1}^{n}\prod_{\substack{g \in \mathcal{P}(i) \\ g \mid f}} g^{v_g(f)}.
    \end{equation*}
\end{definition}

By using splitting partitions, we can give a combinatorial description of irreducible factors of $f_*$ and $f^*$. This amounts to imposing additional conditions on the splitting partitions admitted by $f_*$ and $f^*$.
\begin{definition}[c.f. {\cite[Definition 4.9]{Park2025}}]
    We define the following set of splitting partitions.
    \begin{itemize}
        \item $\Lambda_{n,w,Q}^{\text{for}}$: the set of splitting partitions $\lambda_{n,Q}$ with respect to $(n,w)$ such that $\lambda_{i,j,k,Q} = 0$ whenever $i > \mathfrak{n}$. We use elements $\eta \in \Lambda_{n,w,Q}^{\text{for}}$ to describe splitting partitions that $f_*$ admit.
        \item $\Lambda_{n,w,Q}^{\text{la}}$: the set of splitting partitions $\lambda_{n,Q}$ with respect to $(n,w)$ such that 
        \begin{itemize}
            \item $\lambda_{i,j,k,R} = 0$ whenever $i \leq \mathfrak{n}$
            \item $\lambda_{i,j,0,R} \neq 0$ for some $i > \mathfrak{n}$ and $j \not\equiv 0 \pmod \ell$.
        \end{itemize}
        We use elements $\lambda \in \Lambda_{n,w,Q}^{\text{la}}$ to describe splitting partitions that $f^*$ admit.
    \end{itemize}
\end{definition}

Using two types of splitting partitions defined above, we can define the following subsets of $\mathcal{M}(n;Q,A)$ as follows.
\begin{definition} \label{defn:M(lambda,eta)}
    Let $n > N, w > w'$ be four positive integers.
    \begin{itemize}
        \item We denote by $\mathcal{M}(n,w:N,w'; Q, A)$ the set of polynomials in $\mathcal{M}(n; Q, A)$ satisfying the following two conditions:
        \begin{itemize}
            \item $\deg f = n$ and $\omega(f) = w$.
            \item $\deg f^* = N$, $\omega(f^*) = w'$, and $f^*$ is $\ell$-th power free.
        \end{itemize}
        \item Given two partitions $\lambda \in \Lambda_{N,w',Q}^{\text{la}}$ and $\eta \in \Lambda_{n-N,w-w',Q}^{\text{for}}$, we denote by $\mathcal{M}(\lambda, \eta; Q, A)$ the set of polynomials in $\mathcal{M}(n,w:N,w'; Q, A)$ such that
        \begin{itemize}
            \item $f^*$ admits a splitting partition $\lambda$ with respect to $(N,w')$.
            \item $f_*$ admits a splitting partition $\eta$ with respect to $(n-N,w-w')$.
        \end{itemize}
        \item We note that the following relation holds:
        \begin{equation*}
            \mathcal{M}(n,w:N,w'; Q, A) = \bigsqcup_{\lambda \in \Lambda_{N,w',Q}^{\emph{la}}} \bigsqcup_{\eta \in \Lambda_{n-N,w-w',Q}^{\emph{for}}} \mathcal{M}(\lambda, \eta ; Q, A).
        \end{equation*}
    \end{itemize}
\end{definition}

The theorem below shows that we can obtain the desired changes in algebraic ranks of elliptic curves with respect to subsets of cyclic prime extensions over $K$ by observing these changes over the sets $\mathcal{M}(\lambda,\eta; Q, A)$.
\begin{theorem}[c.f. {\cite[Proposition 4.14]{Park2025}}] \label{theorem:decomposition}
    Let $B \in (0,1)$. Suppose $n$ is a positive number such that 
    \begin{equation*}
        n > \max\{\exp(\exp(\exp(e))), 6 \cdot (\ell^3 + g_{E[\ell]})\}.
    \end{equation*}
    Given a monic polynomial $Q$ and a polynomial $A$ of degree less than $\deg Q$, we have
    \begin{align}
    \begin{split}
        & \hspace{15pt} \#\mathcal{M}(n ; Q, A) - \sum_{w = \lfloor(1-B) \log n \rfloor}^{\lceil(1+B) \log n \rceil} \hspace{5pt} \sum_{w' = \lfloor(1 - \frac{1}{\log \log \log n}) w \rfloor}^{w} \hspace{5pt} \sum_{N = w' \mathfrak{n}}^n \#\mathcal{M}(n,w:N,w'; Q, A) \\
        &\ll_{Q,q} \frac{\# \mathcal{M}(n ; Q, A)}{B^2 \log(2n)}.
    \end{split}
    \end{align}
\end{theorem}
\begin{proof}
    The proof follows from a line by line adaptation of \cite[Proposition 4.14]{Park2025}, which uses large deviation estimates on number of distinct irreducible factors of monic polynomials and Theorem \ref{thm:prop4.13_park}. In place of large deviation estimates, we instead use Lemma \ref{Lem:TuranAP}, which is an analogue of Tur\`an's theorem for $\omega$ on $\mathcal{M}(n ; Q, A)$.
\end{proof}

\subsection{Auxiliary places}

Using Theorem \ref{theorem:decomposition}, we shall reduce the problem of obtaining statistics of rank growths of elliptic curves with respect to cyclic Galois extensions $K(\sqrt[\ell]{f})/K$ as $f$ varies over $\mathcal{M}(n ; Q, A)$ to the same type of problems where $f$ varies over $\mathcal{M}(n, w: N, w'; Q, A)$ given choices of $w, w',$ and $N$. We can specify our analysis further by considering the same type of problems where $f$ varies over $\mathcal{M}(\lambda, \eta; Q, A)$ with fixed choices of splitting partitions $\lambda \in \Lambda_{N,w',Q}^{\emph{la}}$ and $\eta \in \Lambda_{n-N, w-w',Q}^{\emph{la}}$. However, as we will demonstrate (much as in \cite{Park2025}), we require additional notions in order to effectively quantify such changes in ranks. The following subsection, which makes modifications to several notations appearing in \cite[Section 5.2]{Park2025}, aims to introduce these desired additional notions which will allow us to understand rank growths of elliptic curves with respect to cyclic Galois extensions $K(\sqrt[\ell]{f})/K$ as $f$ varies over $\mathcal{M}(\lambda, \eta; Q, A)$.

\begin{definition}
    Given a set $X$, we denote by $\mathrm{PConf}_n(X)$ the ordered configuration set of $n$ elements in $X$:
    \begin{equation*}
        \mathrm{PConf}_n(X) := \{(x_1,\cdots,x_n) \in X^{n} : x_i \neq x_j \text{ for all } 1 \leq i < j \leq n\}.
    \end{equation*}
    Note that $\mathrm{PConf}_n(X)$ admits a transitive permutation action by the symmetric group $S_n$. We write $\mathrm{Conf}_n(X)$ for the quotient set (or the unordered configuration set of $n$ elements in $X$):
    \begin{equation*}
        \mathrm{Conf}_n(X) := \mathrm{PConf}_n(X)/S_n.
    \end{equation*}
\end{definition}

The following lemma shows that the set $\mathcal{M}(\lambda,\eta; Q, A)$ introduced in Definition \ref{defn:M(lambda,eta)} can be identified with a Cartesian product of unordered configuration sets.
\begin{lemma} \label{lemma:characterization_M(lambda,eta)}
    Let $n > N$, $w > w'$ be four integers. Choose splitting partitions $\lambda \in \Lambda_{N,w',Q}^{\emph{la}}$ and $\eta \in \Lambda_{n-N,w-w',Q}^{\emph{for}}$. Then either $\mathcal{M}(\lambda,\eta; Q, A)$ is empty, or there exists a bijection of sets
    \[
        \mathcal{M}(\lambda, \eta; Q, A) \cong \left(\prod_{i,j,k,R} \mathrm{Conf}_{\lambda_{i,j,k,R}}(\mathcal{P}_k(i; Q,R)) \right) \times \left(\prod_{\hat{i}, \hat{j}, \hat{k}, \hat{R}} \mathrm{Conf}_{\eta_{\hat{i},\hat{j},\hat{k},\hat{R}}}(\mathcal{P}_{\hat{k}}(\hat{i}; Q,\hat{R})) \right).
    \]
\end{lemma}
\begin{proof}
    For any $f \in \mathcal{M}(\lambda, \eta; Q, A)$, we have $f \equiv A \pmod Q$. Then the set $\mathcal{M}(\lambda,\eta; Q, A)$ is non-empty if and only if 
    \[
        \left(\prod_{i,j,k} \prod_{\substack{R \in \mathbb{F}_q[t] \\ \deg R < \deg Q}} R^{\lambda_{i,j,k,R}} \right)\left( \prod_{\hat{i}, \hat{j}, \hat{k}} \prod_{\substack{\hat{R} \in \mathbb{F}_q[t] \\ \deg \hat{R} < \deg Q}} \hat{R}^{\eta_{\hat{i},\hat{j},\hat{k},\hat{R}}} \right) \equiv A \pmod Q.
    \]
    So we can factorize every $f \in \mathcal{M}(\lambda,\eta;Q,A)$ as
    \begin{equation*}
        f = f^* f_* = \left( \prod_{i,j,k,R} (g_{i,j,k,R})^{\lambda_{i,j,k,R}} \right) \left( \prod_{\hat{i},\hat{j},\hat{k},\hat{R}} (h_{\hat{i},\hat{j},\hat{k},\hat{R}})^{\eta_{\hat{i},\hat{j},\hat{k},\hat{R}}} \right)
    \end{equation*}
    for some irreducible polynomials $g_{i,j,k,R}$ and $h_{\hat{i},\hat{j},\hat{k},\hat{R}}$ such that
    \begin{itemize}
        \item $g_{i,j,k,R} \in \mathcal{P}_k(i; Q, R)$ and $v_{g_{i,j,k,R}}(f) = j$. Given $i,j,k,R$, there are $\lambda_{i,j,k,R}$ many such $g_{i,j,k,R}$'s.
        \item $h_{\hat{i},\hat{j},\hat{k},\hat{R}} \in \mathcal{P}_{\hat{k}}(\hat{i}; Q, \hat{R})$ and $v_{h_{\hat{i},\hat{j},\hat{k},\hat{R}}}(f) = \hat{j}$. Given $\hat{i},\hat{j},\hat{k},\hat{R}$, there are $\eta_{\hat{i},\hat{j},\hat{k},\hat{R}}$ many such $h_{\hat{i},\hat{j},\hat{k},\hat{R}}$'s.
    \end{itemize}
    The set bijection between $\mathcal{M}(\lambda,\eta; Q, A)$ and the Cartesian product of unordered configuration sets can be defined by
    \begin{equation*}
        f \mapsto \left(\prod_{i,j,k,R} \left\{ g \in \mathbb{F}_q[t] : \substack{g \in \mathcal{P}_k(i;Q,R),\\ v_g(f) = j}\right\} \right) \times \left( \prod_{\hat{i},\hat{j},\hat{k},\hat{R}} \left\{h \in \mathbb{F}_q[t] : \substack{h \in \mathcal{P}_{\hat{k}}(\hat{i}; Q, \hat{R}), \\ v_h(f) = \hat{j}} \right\}\right).
    \end{equation*}
\end{proof}

Using this identification of $\mathcal{M}(\lambda,\eta; Q,A)$, we can describe certain projection maps between such Cartesian products of configuration sets.

\begin{definition}[c.f. {\cite[Definition 5.8, 5.12]{Park2025}}] \label{defn:important_maps}
    Let $n > N$ and $w > w'$ be positive integers. Fix a monic polynomial $Q$ and a polynomial $A$ over $\mathbb{F}_q$. Choose splitting partitions $\lambda \in \Lambda_{N,w',Q}^{\text{la}}$ and $\eta \in \Lambda_{n-N,w-w',Q}^{\text{for}}$.
    \begin{itemize}
        \item Let $d_a$ be the maximal integer for which $\lambda_{i,j,0,R} = 0$ whenever $i > d_a$. An auxiliary place of $f$ is an irreducible factor $g \in \mathcal{P}_0(d_a)$ of $f$.
        \item The auxiliary factor of $f$, denoted by $f_a$, is defined by
        \begin{equation*}
            f_a := \prod_{g \mid f, g \in \mathcal{P}_0(d_a)} g^{v_g(f)},
        \end{equation*}
        where $v_g(f)$ is the multiplicity of $g$ dividing $f$.
        \item We write
        \begin{equation*}
            d_{a^*} := \deg f_a = d_a \cdot \sum_{j=1}^{p-1} \sum_{\substack{R \in \mathbb{F}_q[t] \\ \deg R < \deg Q}} \lambda_{d_a,j,0,R}.
        \end{equation*}
        \item Let $\phi_{d_a}$ be the projection map
        \begin{align}
            \begin{split}
                \mathcal{M}(\lambda,\eta;Q,A) \to \left(\prod_{\substack{i,j,k,R \\ (i,k) \neq (d_a,0)}} \mathrm{Conf}_{\lambda_{i,j,k,R}}(\mathcal{P}_k(i; Q,R)) \right) \times \left(\prod_{\hat{i}, \hat{j}, \hat{k}, \hat{R}} \mathrm{Conf}_{\eta_{\hat{i},\hat{j},\hat{k},\hat{R}}}(\mathcal{P}_{\hat{k}}(\hat{i}; Q,\hat{R})) \right),
            \end{split}
        \end{align}
        which forgets the auxiliary factor $f_a$ of $f$. Note that we have the decomposition 
        \begin{equation*}
            \mathcal{M}(\lambda,\eta;Q,A) = \bigsqcup_{h \in \mathcal{M}(n-d_{a^*})} \phi_{d_a}^{-1}(h).
        \end{equation*}
        \item Fix three integers $d^* \neq d_a$, $1 \leq j^* \leq \ell-1$, $0 \leq k^* \leq 2$, and a polynomial $R^*$ over $\mathbb{F}_q$ of degree less than $\deg Q$ such that $\lambda_{d^*,j^*,k^*,R^*} \neq 0$. Let $\phi_{d^*,j^*,k^*,R^*}$ be the projection map
        \begin{align}
        \begin{split}
            \phi_{d_a}(\mathcal{M}(\lambda,\eta;Q,A)) &\to \left(\prod_{\substack{i,j,k,R \\ (i,k) \neq (d_a,0), \\ (i,j,k,R) \neq (d^*,j^*,k^*,R^*)}} \mathrm{Conf}_{\lambda_{i,j,k,R}}(\mathcal{P}_k(i; Q,R)) \right) \\
            & \hspace{15pt} \times \left(\prod_{\hat{i}, \hat{j}, \hat{k}, \hat{R}} \mathrm{Conf}_{\eta_{\hat{i},\hat{j},\hat{k},\hat{R}}}(\mathcal{P}_{\hat{k}}(\hat{i}; Q,\hat{R})) \right),
        \end{split}
        \end{align}
        which forgets the irreducible factors of $f/f_a$ (i.e., the image of $f$ with respect to $\phi_{d_a}$) lying in the unordered configuration set $\mathrm{Conf}_{\lambda_{d^*,j^*,k^*,R^*}}(\mathcal{P}_{k^*}(d^*; Q, R^*))$.
        \item Let $D^* := n - d_{a^*} - d^* \cdot j^* \cdot \lambda_{d^*,j^*,k^*,R^*}$. Then for any $h \in \mathcal{M}(D^*)$ such that $h \mid f$ for some $f \in \mathcal{M}(\lambda,\eta;Q,A)$, we have a set bijection 
        \begin{equation*}
            (\phi_{d^*,j^*,k^*,R^*} \circ \phi_{d_{a}})^{-1}(h) \cong \mathrm{Conf}_{\lambda_{d^*,j^*,k^*,R^*}}(\mathcal{P}_{k^*}(d^*; Q, R^*)) \times \prod_{j=1}^{\ell-1} \prod_{\substack{R \in \mathbb{F}_q[t] \\ \deg R < \deg Q}} \mathrm{Conf}_{\lambda_{d_a,j,0,R}}(\mathcal{P}_0(d_a;Q,R)).
        \end{equation*}
    \end{itemize}
\end{definition}

\section{variations in ranks of elliptic curves} \label{sec:changes_rank}

Using Tur\'an's theorem and the anatomy of irreducible polynomials over $\mathbb{F}_q$ from the previous section, we obtain lower bounds for the probability that an elliptic curve does not gain rank with respect to base change over $\mathbb{Z}/\ell \mathbb{Z}$ Galois extensions over $K$ satisfying a given congruence relation.

\subsection{Weil restrictions}

We first introduce an auxiliary abelian variety over $K$, whose algebraic rank can be utilized to understand differences in algebraic ranks of elliptic curves with respect to cyclic prime extensions over $K$. The interested reader may consult \cite{MRS07, MR07} for a detailed discussion on the properties of these auxiliary abelian varieties.
\begin{definition}
Given an $\ell$-th power free polynomial $f$ over $\mathbb{F}_q$, let $L_f := K(\sqrt[\ell]{f})$. There exists an order $\ell$ character $\chi_f: \mathrm{Gal}(\overline{K}/K) \to \mathbb{Z}/\ell \mathbb{Z}$ such that the fixed field of its kernel is the field $L_f$. Let $\sigma_f$ be a generator of the Galois group $\mathrm{Gal}(L_f/K)$.

Denote by $E^{\chi_f}$ the $\ell-1$ dimensional abelian variety over $K$ defined as
\begin{equation*}
    E^{\chi_f} := \text{Ker} \left( \text{Tr}_{K}^{L_f}: \text{Res}_K^{L_f} E \to E \right),
\end{equation*}
where $\text{Res}_K^{L_f} E$ is the Weil restriction of scalars of $E$ with respect to the Galois extension $L_f := K(\sqrt[\ell]{f})$, and $\text{Tr}_K^{L_f}(P) := \sum_{j=1}^\ell (\sigma_f)^j \cdot P$ is the map induced from the field trace map.
\end{definition}

We recall some crucial properties of these auxiliary abelian varieties from \cite{MR07}.
\begin{proposition}\label{prop:aav-MR07}
    Let $f$ be an $\ell$-th power free polynomial over $\mathbb{F}_q$. Denote by $\sigma_f$ an order $\ell$ element of $\mathrm{Gal}(L_f/K)$. Then the following hold.
    \begin{enumerate}
        \item $\mathrm{rank}_\mathbb{Z} E^{\chi_f}(K) = \mathrm{rank}_\mathbb{Z} E(L_f) - \mathrm{rank}_\mathbb{Z} E(K)$.
        \item There is a $\mathrm{Gal}(\overline{K}/K)$-equivariant isomorphism $E^{\chi_f}[1-\sigma_f] \cong E[\ell]$.
    \end{enumerate}
\end{proposition}
\begin{proof}
    These two statements follow from \cite[Proposition 4.1, Section 3]{MR07}. 
\end{proof}

\begin{remark}
    % We give an additional exposition regarding Proposition \ref{prop:aav-MR07}. 
    Given an elliptic curve $E$ over $K$, the Weil restriction of scalars $\text{Res}_K^{L_f} E$ can be viewed as an $\ell$-dimensional abelian variety over $K$. Indeed, suppose that we have the Weierstrass equation $E: y^2 = x^3 + Ax + B$. We can substitute 
    \begin{equation*}
        \begin{cases}
        x = x_1 + x_2 \sqrt[\ell]{f} + x_3 \sqrt[\ell]{f^2} + \cdots x_{\ell} \sqrt[\ell]{f^{\ell-1}}, \\
        y = y_1 + y_2 \sqrt[\ell]{f} + y_3 \sqrt[\ell]{f^2} + \cdots y_{\ell} \sqrt[\ell]{f^{\ell-1}},
    \end{cases} 
    \end{equation*}
    into the Weierstrass equation of the elliptic curve $E$. This yields the equation 
    \begin{equation*}
        \sum_{j=1}^\ell g_j(x_1, \cdots, x_n, y_1, \cdots, y_n) \cdot \sqrt[\ell]{f^j} = 0.
    \end{equation*}
    We may therefore regard $\text{Res}_K^{L_f} E$ as a variety defined by a system of $\ell$ polynomial equations $\{g_j\}_{j=1}^\ell$ over $K$ in the $2\ell$ variables $x_1, \cdots, x_\ell, y_1, \cdots, y_\ell$. In the case where $\ell = 2$, for example, we have
    \begin{equation*}
        \begin{cases}
            g_1(x_1,x_2,y_1,y_2) := y_1^2 + y_2^2 - x_1^3 + 3x_1 x_2^2 - Ax_1 - B\\
            g_2(x_1,x_2,y_1,y_2) := 2y_1 y_2 - 3x_1^2 x_2 + x_2^3 - Ax_2.
        \end{cases}
    \end{equation*}
    Using the trace map $\mathrm{Tr}_K^{L_f}$, we obtain a short exact sequence
    \begin{equation*}
        0 \to E^{\chi_f}(K) \to E(L_f) \to E(K).
    \end{equation*}
   Note that the image of the trace map contains $\ell E(K)$ because the action of $\sigma_f$ on $P$ is trivial and so every point $P \in E(K) \subseteq E(L_f)$ satisfies $\mathrm{Tr}_K^{L_f}(P) = \ell P$. From this one deduces the first statement of Proposition \ref{prop:aav-MR07}. When $\ell = 2$, we can show that $E^{\chi_f}(K)$ is the set of $K$-rational points of the quadratic twist $E_f: fy^2 = x^3 + Ax + B$. To see this, note that $E^{\chi_f}(K)$ consists of points $P := (x_1 + x_2 \sqrt{f}, y_1 + y_2 \sqrt{f}) \in E(K(\sqrt{f}))$ such that $P = -\sigma_f P$. This yields
    \begin{equation*}
        (x_1 + x_2 \sqrt{f}, y_1 + y_2 \sqrt{f}) = (x_1 - x_2 \sqrt{f}, -y_1 + y_2 \sqrt{f}).
    \end{equation*}
    Hence, $x_2 = 0$ and $y_1 = 0$, which implies $P = (x_1, y_2 \sqrt{f})$. Plugging this into the Weierstrass equation for $E$ over $L_f$ gives $fy_2^2 = x_1^3 + Ax_1 + B$, which is the Weierstrass equation for $E_f$ over $K$.
\end{remark}

One of the widely used techniques for studying the algebraic rank of an abelian variety $A$ is computing what is called the $\lambda$-Selmer group of $A$, where $\lambda: A \to A$ is an isogeny. Consider the short exact sequence obtained from the isogeny $\lambda$:
\begin{equation*}
    0 \to A[\lambda] \to A \to A \to 0.
\end{equation*}
Taking long exact sequences of cohomology groups, we obtain
\begin{equation*}
    \begin{tikzcd}
    0 \arrow[r] & \frac{A(K)}{\lambda A(K)} \arrow[r] \arrow[d] & H^1_{\et}(K, A[\lambda]) \arrow[r] \arrow[d] & H^1_{\et}(K, A)[\lambda] \arrow[r] \arrow[d] & 0 \\
    0 \arrow[r] & \prod_v \frac{A(K_v)}{\lambda A(K_v)} \arrow[r] & \prod_v H^1_{\et}(K_v, A[\lambda]) \arrow[r] & \prod_v H^1_{\et}(K_v, A)[\lambda] \arrow[r] & 0,
    \end{tikzcd}
\end{equation*}
where the product ranges over all places $v$ of $K$. The $\lambda$-Selmer group $\mathrm{Sel}_\lambda(A/K)$ of $A$ is defined by
\begin{equation*}
    \mathrm{Sel}_\lambda(A/K) := \biggl\{c \in H^1_{\et}(K, A[\lambda]) : \substack{\mathrm{res}_v(c) \in \textrm{Image}\left( \frac{A(K_v)}{\lambda A(K_v)} \to H^1_{\et}(K_v, A[\lambda]) \right) \\ \text{ for all places } v \text{ of } K} \biggr\}.
\end{equation*}
We will use this Selmer group for our families of auxiliary abelian varieties $A := E^{\chi_f}$ with $\lambda := 1 - \sigma_f$. In view of the $\mathrm{Gal}(\overline{K}/K)$-equivariant isomorphism $E^{\chi_f}[1-\sigma_f] \cong E[\ell]$, we can consider the $1-\sigma_f$ Selmer group for our auxiliary abelian varieties.
\begin{definition}
    Given a polynomial $f$ over $\mathbb{F}_q$ which is not of $\ell$-th power, let $\sigma_f$ be a generator of $\mathrm{Gal}(L_f/K) \cong \mathbb{Z}/\ell \mathbb{Z}$. The $1-\sigma_f$ Selmer group of $E^{\chi_f}$ is defined as
    \begin{equation*}
        \mathrm{Sel}_{1-\sigma_f}(E^{\chi_f}/K) := \biggl\{c \in H^1_{\et}(K, E[\ell]) : \substack{\mathrm{res}_v(c) \in \textrm{Image}\left( \frac{E^{\chi_f}(K_v)}{(1-\sigma_f)E^{\chi_f}(K_v)} \to H^1_{\et}(K_v, E[\ell]) \right) \\ \text{ for all places } v \text{ of } K} \biggr\}.
    \end{equation*}
\end{definition}
By \cite[Propositions 2.1, 6.3]{MR07}, the dimension of $\mathrm{Sel}_{1-\sigma_f}(E^{\chi_f}/K)$ upper-bounds the difference between the ranks of $E(L_f)$ and $E(K)$:
\begin{equation*}
    \mathrm{rank}(E/K(\sqrt[\ell]{f})) - \mathrm{rank}(E/K) \leq (\ell - 1) \dim_{\mathbb{F}_\ell} \mathrm{Sel}_{1-\sigma_f}(E^{\chi_f}/K).
\end{equation*}
Consequently, partial information about the difference between the ranks of $E(K(\sqrt[\ell]{f}))$ and $E(K)$ can be obtained from the distribution of the dimension of $\mathrm{Sel}_{1-\sigma_f}(E^{\chi_f}/K)$. In particular, we shall deduce Theorem \ref{theorem:main1} from the following density result on $\dim_{\mathbb{F}_\ell} \mathrm{Sel}_{1-\sigma_f}(E^{\chi_f}/K)$.
% Our goal to understand differences between ranks of $E(K(\sqrt[\ell]{f}))$ and $E(K)$ can hence be partially understood from computing the distribution of dimensions of $\mathrm{Sel}_{1-\sigma_f}(E^{\chi_f}/K)$. In this section, we aim to prove the following theorem, which we will use to prove Theorem \ref{theorem:main1}.
\begin{theorem} \label{theorem:main2}
    Assume the same conditions from Theorem \ref{theorem:main1}. Let $\sigma_f$ be a generator of the Galois group $\mathrm{Gal}(K(\sqrt[\ell]{f})/K)$. Then for any non-negative integer $d \geq 0$,
    \begin{equation*}
        \lim_{n \to \infty} \frac{\#\{f \in \mathcal{M}(n;Q,A) : \dim_{\mathbb{F}_\ell} \mathrm{Sel}_{1-\sigma_f}(E^{\chi_f}/K) = d \}}{\# \mathcal{M}(n;Q,A)} = \prod_{i=0}^\infty \frac{1}{1+\ell^{-i}} \prod_{j=1}^d \frac{\ell}{\ell^j-1}.
    \end{equation*}
\end{theorem}

\subsection{Local Selmer structures}

As in \cite{KMR14, Park2025}, we will prove Theorem \ref{theorem:main2} by studying the distribution of the \textit{the local Selmer group} (or \textit{the local Selmer structure}) of $E$ associated to $\chi_f$. It is a finite dimensional subspace of $H^1_{\et}(K, E[\ell])$ which keeps track of local conditions specified by maximal isotropic subspaces of $H^1_{\et}(K_v, E[\ell])$ over all places $v$ of $K$ parametrized by Cartesian product of local $\mathbb{Z}/\ell \mathbb{Z}$ characters of $\mathrm{Gal}(\overline{K}_v/K_v)$. We refer the reader to \cite[Sections 4, 5]{Park2025} for the definitions and notation regarding the anatomy of monic polynomials of degree $n$, Cartesian products of local characters, and Selmer structures.
\begin{definition}[cyclic order $\ell$ characters in $\text{Hom}(\text{Gal}(\overline{K_v}/K_v),\mu_\ell)$, {\cite[Definition 5.1]{Park2025}}]\label{definition:basic_local_def}
Let $\Sigma$ be a set of places of $K$ that includes the places of bad reduction of $E$, and write $\Sigma_E$ for the set consisting exactly of the places of bad reduction of $E$. Let $\sigma$ be a square-free product of places coprime to elements in $\Sigma$.
\begin{itemize}
    \item $\Omega_\sigma$: the set of finite Cartesian products of local characters
    \begin{equation*}
        \chi := (\chi_v)_v \in \Omega_\sigma := \prod_{v \in \Sigma(\sigma)} \text{Hom}(\Gal(\overline{K}_v/K_v), \mu_\ell)
    \end{equation*}
    such that the component $\chi_v$ is ramified if $v \mid \sigma$. We will denote by $\text{Hom}_{unr}(\Gal(\overline{K}_v/K_v), \mu_\ell)$ the set of unramified local characters at place $v$, and by $\text{Hom}_{ram}(\Gal(\overline{K}_v/K_v), \mu_\ell)$ the set of ramified local characters at place $v$. Assuming that $\mu_\ell \subseteq K_v$, there are $\ell$ distinct unramified local characters at $v$, and $\ell(\ell-1)$ distinct ramified local characters at $v$.
    \item $\Omega_E$: the set of finite Cartesian products of local characters
    \begin{equation*}
        \chi := (\chi_v)_v \in \Omega_E := \prod_{\substack{v \in \Sigma_E }} \text{Hom}(\Gal(\overline{K}_v/K_v), \mu_\ell).
    \end{equation*}

    \item Fix an element $\chi \in \Omega_\sigma$. Let $\mathfrak{v}$ be a place over $K$ such that $\mathfrak{v} \not\in \Sigma(\sigma)$. Let $\chi' \in \Omega_{\sigma \mathfrak{v}}$ be an element such that
    \begin{itemize}
    \item For any $v \in \Sigma(\sigma)$, $\chi'_v = \chi_v$.
    \item At $\mathfrak{v}$, $\chi'_\mathfrak{v}$ is ramified.
    \end{itemize} 
    Denote by $\Omega_{\chi,\mathfrak{v}}$ the set of local characters $\chi'$ satisfying the two conditions above. Note that
    \begin{equation*}
    \Omega_{\sigma \mathfrak{v}} = \bigsqcup_{\chi \in \Omega_\sigma} \Omega_{\chi,\mathfrak{v}}.
    \end{equation*}
\end{itemize}
\end{definition}

\begin{definition}[{\cite[Definition 5.2]{Park2025}}]\label{definition:consecutive_local_char}
We introduce the following notations on local Selmer groups of $E$ associated to Cartesian products of cyclic order $\ell$ characters over all places $v$ of $K$, denoted by $\chi := (\chi_v)_v \in \prod_{v \text{ place of } K}\text{Hom}(\text{Gal}(\overline{K_v}/K_v,\mu_\ell)$. Denote by $K_v^{\chi_v}$ the fixed field of $\mathrm{Ker}(\chi_v)$. Let $\sigma_{\chi_v}$ be the generator of the Galois group $\mathrm{Gal}(K^{\chi_v}/K_v)$, and we use the abbreviation $\pi := 1 - \sigma_{\chi_v}$, which is an isogeny over the twist $1-\sigma_{\chi_v}: E^{\chi_v} \to E^{\chi_v}$.
\begin{itemize}
\item Given a Cartesian product of local characters $\chi \in \Omega_\sigma$, we define the local Selmer group of $E$ associated to $\chi$ by
    \begin{equation*}
        \Sel(E[\ell], \chi) := \text{Ker} \left( H^1_{\et}(K, E[\ell]) \to \prod_v H^1_{\et}(K_v, E[\ell])/\mathcal{H}^{\chi}_v \right),
    \end{equation*}
    where
    \begin{equation} \label{eqn:defn-sbsp}
        \mathcal{H}^{\chi}_v := \begin{cases}
            \text{im} \left(\delta_v^\chi: E^{\chi_v}(K_v)/\pi E^{\chi_v}(K_v) \to H^1(K_v, E[\ell]) \right) &\text{ if } v \in \Sigma(\sigma),\\
            H^1(\mathcal{O}_{K_v}, E[\ell]) &\text{ if } v \not\in \Sigma(\sigma).
        \end{cases}
    \end{equation}
    Assume that $E$ satisfies all but the third condition from Condition \ref{condition}. 
    %the following conditions hold:
    %\begin{align} \label{equation:assumption_local_twists}
%\begin{split}
%    & \bullet E \text{ is a non-isotrivial elliptic curve over } K. \\ 
%    & \bullet E \text{ has a place of split multiplicative reduction}. \\
    % & \bullet \text{The constant field } \F_q \text{ has characteristic coprime to } 2,3,p, \text{ and contains } \mu_p. \\
%    & \bullet \text{The image of } \Gal(\overline{K}/K) \to \text{Aut}(E[\ell]) \text{ contains } \SL_2(\mathbb{F}_\ell). 
%\end{split}
%\end{align}
Under these conditions, we use the isomorphism
    \begin{align*}
        H^1_{\et}(K, E[\ell]) &\cong H^1_{\et}(K, E^\chi[\pi]), \\
        H^1_{\et}(K_v, E[\ell]) &\cong H^1_{\et}(K_v, E^{\chi_v}[\pi]),
    \end{align*}
    to define the local Selmer group $\Sel(E[\ell], \chi)$. 
    \item Recall that the Weil pairing $E[\ell] \times E[\ell] \to \mu_\ell$ and the cup product on $H^1_{\et}(K_v, E[\ell])$ induce a symmetric pairing
    \begin{equation*}
        H^1_{\et}(K_v, E[\ell]) \times H^1_{\et}(K_v,E[\ell]) \to \mathbb{F}_\ell.
    \end{equation*}
    Denote by $q_v$ the quadratic form induced from the symmetric pairing stated above. Then $\mathcal{H}_v^\chi$ is a maximal isotropic subspace of $H^1_{\et}(K_v,E[\ell])$ with respect to $q_v$. Furthermore, if $v \in \Sigma(\sigma) \setminus \Sigma$, then $\mathcal{H}_v^\chi \cap H^1(\mathcal{O}_{K_v}, E[\ell]) = 0$. 
    \item If $v \in \mathcal{P}_0$, then $\mathcal{H}_v^{\chi}$ is trivial. If $v \in \mathcal{P}_1 \cap \Sigma(\sigma)$, then there is a unique $1$-dimensional ramified subspace, denoted by $\mathcal{H}^1_{ram}$. If $v \in \mathcal{P}_2 \cap \Sigma(\sigma)$, then there are $\ell$ distinct $2$-dimensional ramified subspaces $\mathcal{H}_v^{\chi}$, each corresponding to a tamely totally ramified cyclic $p$ extension $\overline{K}_v^{\text{Ker}(\chi_v)}$ over $K_v$. For such a $v$ we have a set bijection 
    \begin{equation*}
        \alpha_v: \frac{\text{Hom}_{ram}(\text{Gal}(\overline{K_v}/K_v),\mu_\ell)}{\text{Aut}(\mu_\ell)} \to \{\mathcal{H}_v^\chi\}_{\chi \in \text{Hom}_{ram}(\text{Gal}(\overline{K_v}/K_v),\mu_\ell)}, 
    \end{equation*}
    which leads to the following characterization of the subspace $\mathcal{H}_v^\chi$ defined by (\ref{eqn:defn-sbsp}):
    \begin{equation} \label{eqn:defn-sbsp2}
        \mathcal{H}_v^\chi := \begin{cases}
            \alpha_v(\overline{K}_v^{\text{Ker}(\chi_v)}) &\text{ if } v \in \mathcal{P}_2 \cap \Sigma(\sigma), \\
            \mathcal{H}^1_{ram} &\text{ if } v \in \mathcal{P}_1 \cap \Sigma(\sigma), \\
            0 &\text{ if } v \in \mathcal{P}_0 \cap \Sigma(\sigma), \\
            \text{im} \delta_v^\chi &\text{ if } v \in \Sigma \setminus \mathcal{P} \text{ and } \mathcal{H}_v = \overline{K}_v^{\text{Ker}(\chi_v)},\\
            H^1(\mathcal{O}_{K_v}, E[\ell]) &\text{ if } v \not\in \Sigma(\sigma).
            \end{cases}
    \end{equation}
    \item Given a set of local characters $\chi \in \Omega_\sigma$, denote by $\text{rk}(\chi)$ the dimension of $\Sel(E[\ell],\chi)$ as an $\F_\ell$-vector space. By the identification of $\mathcal{H}_v^\chi$ above, we have $\text{rk}(\chi) = \text{rk}(\chi')$ if the following two conditions are satisfied:
    \begin{itemize}
        \item $\text{Ker}(\chi_v) = \text{Ker}(\chi'_v) \subseteq \text{Gal}(\overline{K}_v/K_v)$ for every $v \in \mathcal{P}_2 \cap \Sigma(\sigma)$.
        \item $\text{rk}(\hat{\chi}) = \text{rk}(\hat{\chi'})$, where $\hat{\chi} := (\chi_v)_{v \in \Sigma_E} \in \Omega_E$ (and likewise for $\hat{\chi'}$).
    \end{itemize}
    Any changes in local conditions over places $v \in \mathcal{P}_0$ do not affect the values of $\text{rk}(\chi)$.
\item Denote by $t_{\chi}(\mathfrak{v})$ the dimension of the image of the local Selmer group $\Sel(E[\ell],\chi)$ with respect to the localization map at $\mathfrak{v}$, i.e.,
\begin{equation*}
    t_{\chi}(\mathfrak{v}) := \text{dim}_{\F_p} \text{im} \left(\text{loc}_\mathfrak{v}: \Sel(E[\ell],\chi) \to H^1(\Oh_{K_\mathfrak{v}},E[\ell])\right).
\end{equation*}
Note that if $\mathfrak{v} \in \mathcal{P}_i$, then $0 \leq t_{\chi}(\mathfrak{v}) \leq i$. Furthermore, $t_\chi(\mathfrak{v}) = t_{\chi'}(\mathfrak{v})$ if $\text{Ker}(\chi_v) = \text{Ker}(\chi'_v) \subseteq \text{Gal}(\overline{K}_v/K_v)$ for every $v \in \Sigma(\sigma)$.
\end{itemize}
\end{definition}

\begin{remark} \label{remark:local-global-selmer}
The notion of local Selmer groups of $E$ encompasses the construction of $\mathrm{Sel}_{1-\sigma_f}(E^{\chi_f}/K)$.
Given a global character $\chi_f \in \mathrm{Hom}(\mathrm{Gal}(\overline{K}/K), \mathbb{Z}/\ell \mathbb{Z})$ for which the fixed field of $\mathrm{Ker}(\chi_f)$ is $K(\sqrt[\ell]{f})$, and a place $v$, we denote by $\chi_{f,v} \in \mathrm{Hom}(\mathrm{Gal}(\overline{K}_v/K_v), \mathbb{Z}/\ell \mathbb{Z})$ the restriction of $\chi_f$ at place $v$. Observe from the definition of $\mathrm{Sel}_{1-\sigma_f}(E^{\chi_f}/K)$ that 
\begin{equation*}
    \mathrm{Sel}_{1-\sigma_f}(E^{\chi_f}/K) = \mathrm{Sel}(E[\ell], (\chi_{f,v})_v).
\end{equation*}
For every place $v$ of $K$, the image of the local Kummer maps $\textrm{Image}\left( \frac{E^{\chi_f}(K_v)}{(1-\sigma_f)E^{\chi_f}(K_v)} \to H^1_{\et}(K_v, E[\ell]) \right)$ is a maximal isotropic subspace of $H^1_{\et}(K_v, E[\ell])$ with respect to the Tate quadratic form.
Furthermore, except for places $v \in \Sigma_f(\overline{f}^*)$, all local conditions $\mathcal{H}_v^\chi$ are equal to unramified cohomology groups $H^1_{ur}(K_v, E[\ell])$. See \cite[Section 4]{PR12} for further details.
\end{remark}

To understand the distribution of local Selmer groups, we first prove the following equidistribution result for Dirichlet characters with modulus $Q$ given a choice of an elliptic curve $E/K$.
\begin{proposition}[c.f. {\cite[Corollary 4.17]{Park2025}}] \label{proposition:equidistribution}
    Let $E$ be an elliptic curve over $K$ satisfying conditions in Condition \ref{condition} and fix $Q\in\M(m)$. Let $K_Q/K$ be the ray class field with modulus $Q + \infty$. Suppose that $h_1, h_2, \cdots, h_w$ are irreducible polynomials over $\mathbb{F}_q$ which are coprime to $Q$. Let $n$ be an integer such that $\sum_{j=1}^w \deg h_j \leq n$ and $w \leq 2 \log n$. Let $i > \mathfrak{n}$ be a positive integer.
    \begin{itemize}
        \item If $\ell \geq 5$ or $K_Q(\sqrt[\ell]{h_1}, \cdots, \sqrt[\ell]{h_w}) \cap K(E[\ell]) = K$, then for any $a \in \mu_\ell^{\oplus w}$ and any polynomial $R$ over $\mathbb{F}_q$ of degree at most $m-1$, there exists a constant $\widehat{C}_{E,Q,q,\ell} > 0$ depending on $E$, $Q$, $q$, and $\ell$ such that
        \begin{equation*}
            \left| \frac{\#\big\{v \in \mathcal{P}_k(i;Q,R) : \left( \left( \frac{v}{h_j} \right)_\ell\right)_{j=1}^w = a \in \mu_\ell^{\oplus w}\big\}}{\# \mathcal{P}_k(i;Q,R)} - \frac{1}{\ell^w} \right| < \widehat{C}_{E,Q,q,\ell} \cdot n^{-2 \log n + 2 \log \ell}.
        \end{equation*}
        \item If $\ell = 2,3$ and $K_Q(\sqrt[\ell]{h_1}, \cdots, \sqrt[\ell]{h_w}) \cap K(E[\ell]) \neq K$, then there are $\ell^w - \ell^{w-1}$ elements $a \in \mu_\ell^{\oplus w}$ such that $\left( \left( \frac{v}{h_j} \right)_\ell\right)_{j=1}^w \neq a$ for all $v \in \mathcal{P}_k(i;Q,R)$. For the remaining $\ell^{w-1}$ elements $a \in \mu_\ell^{\oplus w}$, there exists a constant $\widehat{C}_{E,Q,q,\ell} > 0$ depending on $E$, $Q$, $q$, and $\ell$ such that
        \begin{equation*}
            \left| \frac{\#\big\{v \in \mathcal{P}_k(i;Q,R) : \left( \left( \frac{v}{h_j} \right)_\ell\right)_{j=1}^w = a \in \mu_\ell^{\oplus w}\big\}}{\# \mathcal{P}_k(i;Q,R)} - \frac{1}{\ell^{w-1}} \right| < \widehat{C}_{E,Q,q,\ell} \cdot n^{-2 \log n + 2 \log \ell}.
        \end{equation*}
    \end{itemize}
\end{proposition}
\begin{proof}
    The proposition follows from the effective Chebotarev density theorem (Theorem \ref{thm:effective_chebotarev}), as elaborated in \cite[Corollary 4.17]{Park2025}. When $\ell \ge5$, consider the Galois extension $L/K$, where
    \begin{equation*}
        L = K_Q\big(E[\ell], \sqrt[\ell]{h_1}, \cdots, \sqrt[\ell]{h_w}\big).
    \end{equation*}
    Here we crucially use the fact that if $\ell \geq 5$, then $\mathrm{SL}_2(\mathbb{F}_\ell)$ has no normal subgroup of index $\ell$, which implies
    \begin{equation*}
        \mathrm{Gal}(L/K) = \mathrm{Gal}(K_Q/K) \times \mu_\ell^{\oplus w} \times \mathrm{SL}_2(\mathbb{F}_\ell).
    \end{equation*}
    The constant $\widehat{C}_{E,Q,q,\ell}$ arises from an application of the effective Chebotarev density theorem and the Riemann--Hurwitz theorem to the field extension $L/K$: 
    \begin{equation*}
        \widehat{C}_{E,Q,q,\ell} = 16(\# \mathrm{Gal}(L/K) + g_L) \leq 16 q^m \ell^{w} (\ell^3-\ell) (1 + (\deg \Delta_E + m + w)).
    \end{equation*}
    For the other cases, we use the fact that there exists a unique normal subgroup of $\mathrm{SL}_2(\mathbb{F}_\ell)$ of index $\ell$, which implies
    \begin{equation*}
        \mathrm{Gal}(L/K) = \mathrm{Gal}(K_Q/K) \times \mu_\ell^{\oplus w-1} \times \mathrm{SL}_2(\mathbb{F}_\ell).
    \end{equation*}
    Applying the effective Chebotarev density theorem to the field extension $L/K$, we arrive at the asserted inequality with the same constant $\widehat{C}_{E,Q,q,\ell}$ as defined above.
\end{proof}

With Proposition \ref{proposition:equidistribution} at hand, we examine variations in the dimensions of local Selmer groups with respect to changing local conditions at a single place in arithmetic progression. We first discuss the Chebotarev conditions that determine the values of $t_\chi(\mathfrak{v})$.

\begin{proposition}[c.f. {\cite[Proposition 5.4]{Park2025}}] \label{prop:tchiv_density}
    Let $E$ be an elliptic curve over $K$ satisfying Condition \ref{condition}. Fix $Q\in\M(m)$ of degree $m$ and a square-free product of places $\sigma$ coprime to elements in $\Sigma_E$. Fix $\chi \in \Omega_\sigma$ and define $d_{k,j}$ by
    \begin{equation*}
        d_{k,j} := \begin{cases}
            1 &\text{ if } k =0, j = 0, \\
            1 - {\ell^{-\mathrm{rk}(\chi)}} &\text{ if } k=1, j = -1, \\
            {\ell^{-\mathrm{rk}(\chi)}} &\text{ if } k= 1, j = 1, \\
            1 - (\ell+1) {\ell^{-\mathrm{rk}(\chi)}} + \ell \cdot {\ell^{-2 \cdot \mathrm{rk}(\chi)}} &\text{ if } k=2, j=-2, \\
            (\ell + 1) {\ell^{-\mathrm{rk}(\chi)}} - (\ell+1) {\ell^{-2 \cdot \mathrm{rk}(\chi)}} &\text{ if } k=2, j=0, \\
            {\ell^{-2 \cdot \mathrm{rk}(\chi)}} &\text{ if } k=2, j=2. \\
            0 &\text{ otherwise }.
        \end{cases}
    \end{equation*}
    Let $i > (12 + 2 \max_{\chi \in \Omega_E} \mathrm{rk}(\chi) + 6 \# \Sigma_E(\sigma)) \log \ell$ be a positive integer. Let $R$ be a polynomial of degree at most $m-1$ such that $\#\mathcal{P}_k(i;Q,R) \neq 0$. Then there exists a constant $C_{E,Q,q,\ell} > 0$ depending on $E$, $Q$, $q$, and $\ell$ such that
    \begin{equation*}
        \left| \frac{\#\{\mathfrak{v} \in \mathcal{P}_k(i;Q,R): \mathfrak{v} \not\in \Sigma_E(\sigma) \text{ and } t_\chi(\mathfrak{v}) = j\}}{\#\{\mathfrak{v} \in \mathcal{P}_k(i;Q,R) : \mathfrak{v} \not\in \Sigma_E(\sigma)\}} - d_{k,j} \right| < C_{E,Q,q,\ell} \cdot \ell^{3 \# \Sigma_E(\sigma)} q^{-i/2}.
    \end{equation*}
\end{proposition}
\begin{proof}
    The proof follows from an adaptation of \cite[Proposition 5.4]{Park2025}, but there are a number of changes in its proof that needs to be made.

    Recall the following restriction morphism of cohomology groups obtained from the inflation-restriction sequence:
    \begin{equation*}
        \mathrm{Res}: H^1_{\et}(K, E[\ell]) \hookrightarrow H^1_{\et}(K(E[\ell]), E[\ell])^{\mathrm{Gal}(K(E[\ell])/K)} \cong \mathrm{Hom}(\mathrm{Gal}(\overline{K}/K(E[\ell])), E[\ell])^{\mathrm{Gal}(K(E[\ell])/K)}.
    \end{equation*}
    The fact that the first map is an injection follows from the assumption that $\mathrm{Gal}(K(E[\ell])/K) \cong \mathrm{SL}_2(\mathbb{F}_\ell)$, i.e., $E[\ell]$ is an irreducible $\mathrm{Gal}(\overline{K}/K)$-module. Using this map, we can identify $c \in \mathrm{Sel}(E[\ell],\chi)$ with a homomorphism $c: \mathrm{Gal}(\overline{K}/K(E[\ell])) \to E[\ell]$.

    Let $F_{\sigma,\chi}$ be the fixed field of the following subgroup:
    \begin{equation*}
        \bigcap_{c \in \mathrm{Sel}(E[\ell],\chi)} \mathrm{Ker}\left( \mathrm{Res}(c) \right) \subseteq \mathrm{Gal}(\overline{K}/K(E[\ell])).
    \end{equation*}
    By \cite[Proposition 9.4]{KMR14}, the field $F_{\sigma,\chi}$ is a Galois extension over $K$ unramified outside of places in $\Sigma_E(\sigma)$ such that
    \begin{equation*}
        \mathrm{Gal}(F_{\sigma,\chi}/K(E[\ell])) \cong E[\ell]^{\oplus \mathrm{rk}(\chi)}.
    \end{equation*}
    By \cite[Proposition 5.4, p. 3298, \textit{Constant field of} $F_{\sigma,\chi}$]{Park2025}, the constant fields of $F_{\sigma,\chi}$ and $K$ are equal to each other. For this we use the condition that $E$ has a place of split multiplicative reduction. 

    Using the techniques of the proof of \cite[Proposition 9.4]{KMR14} and \cite[Proposition 5.4, p. 3298, \textit{Frobenius conjugacy class}]{Park2025}, the non-zero values of $d_{k,j}$ from the statement of the proposition are given by ratios of two subsets $S_{k,j}$, $S'_k$ of $\mathrm{Gal}(F_{\sigma,\chi}/K)$ which are stable under conjugation. In particular, we have $d_{k,j} = \# S_{k,j}/\# S'_k\in[0,1]$. One obtains that
    \begin{equation*}
        \begin{cases}
            \mathfrak{v} \in \mathcal{P}_k(i) &\iff \mathrm{Frob}_\mathfrak{v} \in S'_k, \\
            \dim_{\mathbb{F}_\ell} \mathrm{im} \; \mathrm{loc}_{\mathfrak{v}} = j \text{ and } \mathfrak{v} \in \mathcal{P}_k(i) &\iff \mathrm{Frob}_\mathfrak{v} \in S_{k,j}.
        \end{cases}
    \end{equation*}
    To obtain densities for subsets of primes in arithmetic progression $\mathcal{P}_k(i;Q,R)$, we consider the Galois extension $F_{\sigma,\chi} K_Q$ obtained from a compositum of $F_{\sigma,\chi}$ with the ray class field $K_Q$ with modulus $Q + \infty$. There exist two subsets $S^*_{k,j}$ and $S^{'*}_k$ of $\mathrm{Gal}(F_{\sigma,\chi}K_Q/K)$ which are stable under conjugation such that $\# S^*_{k,j} = \# S_{k,j}$, $\# S'_k = \# S^{'*}_k$, and
    \begin{equation*}
        \begin{cases}
            \mathfrak{v} \in \mathcal{P}_k(i;Q,R) &\iff \mathrm{Frob}_\mathfrak{v} \in S^{'*}_k, \\
            \dim_{\mathbb{F}_\ell} \mathrm{im} \; \mathrm{loc}_{\mathfrak{v}} = j \text{ and } \mathfrak{v} \in \mathcal{P}_k(i;Q,R) &\iff \mathrm{Frob}_\mathfrak{v} \in S^*_{k,j}.
        \end{cases}
    \end{equation*}
    We can now apply Theorem \ref{thm:effective_chebotarev} to the field extension $F_{\sigma,\chi} K_Q$ as outlined in \cite[Proposition 5.4, pp. 3298--3299, \textit{Effective error bounds}]{Park2025} to conclude the proof. The only difference from the case $Q = 1$ is the explicit constant for the error term $C_{E,Q,q,\ell}$ which now depends on $q$ and $Q$. This is because as computed in the proof of Proposition \ref{proposition:equidistribution}, the genus of the field extension $K_Q$ depends on $q$ and $\deg Q=m$. Indeed, we have
    \begin{equation*}
        \# \mathrm{Gal}(F_{\sigma,\chi}K_Q/K) \leq \ell^{\max_{\chi \in \Omega_E} \mathrm{rk}(\chi)} \ell^{2 \# \Sigma_E(\sigma)} (\ell^3-\ell) q^{m}.
    \end{equation*}
    By the Riemann--Hurwitz theorem, the field $F_{\sigma,\chi}K_Q$ has genus
    \begin{align}
    \begin{split}
        g_{F_{\sigma,\chi}K_Q} &\leq \ell^{\max_{\chi \in \Omega_E} \mathrm{rk}(\chi)}\ell^{2 \# \Sigma_E(\sigma)} (\ell^3-\ell)  q^{m} (\# \Sigma_E(\sigma) + m).
    \end{split}
    \end{align}
    Hence one obtains (using $\ell^{2 \# \Sigma_E(\sigma)} \# \Sigma_E(\sigma) \leq \ell^{3 \# \Sigma_E(\sigma)}$)
    \begin{equation*}
        \# \mathrm{Gal}(F_{\sigma,\chi}K_Q/K) + g_{F_{\sigma,\chi}K_Q} \leq \ell^{\max_{\chi \in \Omega_E} \mathrm{rk}(\chi)}  q^{m} (\ell^3-\ell) (m+2) \ell^{3 \# \Sigma_E(\sigma)}.
    \end{equation*}
    We apply Theorem \ref{thm:effective_chebotarev} by setting the field extension as $L := F_{\sigma, \chi}K_Q$, $G := \mathrm{Gal}(F_{\sigma, \chi}K_Q/K)$, $g_L := g_{F_{\sigma, \chi}K_Q}$, and the two subsets as $S^{'} := S^{'*}_k$ and $S := S_{k,j}^*$.
    The proof is complete upon taking
    \begin{equation*}
        C_{E,Q,q,\ell} := 16 \ell^{\max_{\chi \in \Omega_E} \mathrm{rk}(\chi)} q^{m} (\ell^3-\ell) (m+2).
    \end{equation*}
\end{proof}

Having obtained the Chebotarev conditions for $t_\chi(\mathfrak{v})$, we can estimate the probability that $\mathrm{rk}(\chi')$ differs from $\mathrm{rk}(\chi)$ by any given amount as $\chi'$ ranges over elements in $\Omega_{\chi, \mathfrak{v}}$ and $\mathfrak{v}$ ranges over elements in $\mathcal{P}_k(i; Q, R) \setminus \Sigma_E(\sigma)$ for sufficiently large $i$.
\begin{proposition}[c.f. {\cite[Proposition 9.5]{KMR14}}] \label{prop:changes_local_rank}
    Let $E$ be an elliptic curve over $K$ satisfying Condition \ref{condition}. Assume all notations and conditions from Proposition \ref{prop:tchiv_density}. Fix $\chi \in \Omega_\sigma$. Given integers $j$ and $0 \leq k \leq 2$, we define
    \begin{equation*}
        m(k,j) := \begin{cases}
            1 &\text{ if } k = 0, \; j = 0, \\
            1 - \ell^{-\mathrm{rk}(\chi)} &\text{ if } k = 1, \; j = -1, \\
            \ell^{-\mathrm{rk}(\chi)} &\text{ if } k = 1, \; j = 1, \\
            1 - (\ell + 1) \ell^{-\mathrm{rk}(\chi)} + \ell \cdot \ell^{-2 \cdot \mathrm{rk}(\chi)} &\text{ if } k = 2, \; j = -2, \\
            (\ell + 1) \ell^{-\mathrm{rk}(\chi)} + (\ell + 1/\ell) \ell^{-2 \cdot \mathrm{rk}(\chi)} &\text{ if } k = 2, \; j = 0, \\
            1/\ell \cdot \ell^{-2 \cdot \mathrm{rk}(\chi)} &\text{ if } k = 2, \; j = 2, \\
            0 &\text{ otherwise}.
        \end{cases}
    \end{equation*}
    Then for any integer $j$, 
    \begin{equation*}
        \left| \sum_{\substack{\mathfrak{v} \in \mathcal{P}_k(i; Q, R) \\ \mathfrak{v} \not\in \Sigma_E(\sigma) }} \frac{\# \{\chi' \in \Omega_{\chi, \mathfrak{v}} : \; \mathrm{rk}(\chi') - \mathrm{rk}(\chi) = j\}}{\# (\mathcal{P}_k(i; Q, R) \setminus \Sigma_E(\sigma)) \cdot \# \Omega_{\chi, \mathfrak{v}}} - m(k,j) \right| < C_{E,Q,q,\ell} \cdot \ell^{3 \# \Sigma_E(\sigma)} q^{-i/2},
    \end{equation*}
    where the constant $C_{E,Q,q,\ell}$ is the same constant appearing in Proposition \ref{prop:tchiv_density}.
\end{proposition}
\begin{proof}
    By \cite[Proposition 7.2]{KMR14} and \cite[Proposition 5.3]{Park2025}, we have, for $\chi' \in \Omega_{\chi, \mathfrak{v}}$ with $\mathfrak{v} \in \mathcal{P}_k(i; Q, R) \setminus \Sigma_E(\sigma)$,
    \begin{equation*}
        \mathrm{rk}(\chi') - \mathrm{rk}(\mathrm{\chi}) = \begin{cases}
            -2 &\text{ if } k = 2, t_\chi(\mathfrak{v}) = 2, \\
            0 &\text{ if } k = 2, t_\chi(\mathfrak{v}) = 1, \\
            0 &\text{ if } k = 2, t_\chi(\mathfrak{v}) = 0 \text{ for } (\ell - 1)^2 \text{ many characters } \chi' \in \Omega_{\chi, \mathfrak{v}}, \\
            2 &\text{ if } k = 2, t_\chi(\mathfrak{v}) = 0 \text{ for } (\ell - 1) \text{ many other characters } \chi' \in \Omega_{\chi, \mathfrak{v}}, \\
            -1 &\text{ if } k = 1, t_\chi(\mathfrak{v}) = 1, \\
            1 &\text{ if } k = 1, t_\chi(\mathfrak{v}) = 0, \\
            0 &\text{ if } k = 0.
        \end{cases}
    \end{equation*}
    The proposition follows now from Proposition \ref{prop:tchiv_density} and the computation that
    \begin{equation*}
        m(k,j) = \begin{cases}
            d_{k,j} &\text{ if } k \leq 1, \text{ or } k = 2, j = 0, \\
            d_{k,1} + (1-1/\ell) \cdot d_{k,2} &\text{ if } k=2, j=1, \\
            1/\ell \cdot d_{k,2} &\text{ if } k=2, j=2.
        \end{cases}
    \end{equation*}
\end{proof}

\subsection{Previous results on Markov operators}

Proposition \ref{prop:changes_local_rank} provides estimates for how likely two local Selmer structures obtained from changing twisting data at a single place differ from each other by a prescribed amount. An effective way to perform combinatorial book-keeping of changes in dimensions of local Selmer structures obtained from changing twisting data at multiple places is by building a stochastic model. We will use the Markov operators defined below to model these changes. The interested reader is referred to \cite[Section 6.1]{Park2025} for a more detailed discussion on these operators.

\begin{definition}[The mod $\ell$ Lagrangian Markov operator, {\cite[Definition 6.1]{Park2025}}]\label{defn:LMO}
Given a prime $\ell$, the \emph{mod $\ell$ Lagrangian Markov operator} $M_L = [m_{r,s}]$ on the state space of non-negative integers $\Z_{\geq 0}$ is defined by
\begin{equation*}
    m_{r,s} = \begin{cases}
    1 - \ell^{-r} &\text{ if } s = r-1 \geq 0, \\
    \ell^{-r} &\text{ if } s = r+1, \\
    0 &\text{ else}.
    \end{cases}
\end{equation*}
\end{definition}

\begin{lemma}[c.f. {\cite[Corollary 6.7]{Park2025}}]\label{corollary:uniform_markov}
Let $\mu: \Z_{\geq 0} \to [0,1]$ be a probability distribution on the (countable) state space $\Z_{\geq 0}$. For each prime $\ell$, define the operator
\begin{equation*}
    M := \left( 1 - \frac{\ell}{(\ell^2-1)} \right)I + \frac{1}{\ell} M_L + \frac{1}{(\ell^3-\ell)} M_L^2,
\end{equation*}
where $I$ is the identity operator on $\Z_{\geq 0}$, and let $\rho$ be the stationary probability distribution of the Markov operator given by
\begin{equation*}
    \rho(z) := \lim_{n\to\infty}(M^n\mu)(z)=\prod_{k=0}^\infty \frac{1}{1 + \ell^{-k}} \cdot \prod_{j=1}^z \frac{\ell}{\ell^j - 1}
\end{equation*}
for every $z\in \mathbb{Z}_{\geq 0}$. Then there exist a constant $\gamma_\ell\in[0,1)$, depending on $\ell$, and a constant $c > 0$ such that for all $n \in \mathbb{N}$,
\begin{equation*}
    \sup_{z \in \mathbb{Z}_{\geq 0}} \left| (M^n\mu)(z) - \rho(z) \right| < c \gamma_\ell^n (\mathbb{E}[\ell^\mu] + 1),
\end{equation*}
where $\mathbb{E}[\ell^\mu] := \sum_{z \in \mathbb{Z}_{\geq 0}} \ell^z \mu(z)$.
\end{lemma}

\subsection{Relating Markov operators with local Selmer structures}

We can extend the description of variations in dimensions of local Selmer groups from the set of irreducible polynomials in arithmetic progression $\mathcal{P}_k(i;Q,R)$ to subsets of polynomials $\mathcal{M}(\lambda,\eta;Q,A)$ by using the maps $\phi_{d_a}$ and $\phi_{d^*,j^*,k^*,R^*}$ from previous sections. But we first recall a few more additional definitions.

\begin{definition}[{\cite[Definition 5.7]{Park2025}}]
For each $f \in \mathcal M(n)$, denote by $\overline{f}$, $\overline{f}_*$, and $\overline{f}^*$ the square-free polynomial over $\F_q$ defined by
\begin{align}
\begin{split}
    \overline{f} := \prod_{\substack{g \mid f \\ g \in \mathcal{P}_1 \cup \mathcal{P}_2}} g, \qquad \overline{f}_* := \prod_{\substack{g \mid f_* \\ g \in \mathcal{P}_1 \cup \mathcal{P}_2}} g, \qquad \overline{f}^* := \prod_{\substack{g \mid f^* \\ g \in \mathcal{P}_1 \cup \mathcal{P}_2}} g.
\end{split}
\end{align}
In other words, they are products of irreducible factors of $f$ (and $f_*$ and $f^*$, respectively) of degree greater than $\mathfrak{n}$ which lies in $\mathcal{P}_1$ or $\mathcal{P}_2$.
\end{definition}

\begin{definition}[{\cite[Definition 5.9]{Park2025}}]
For each $f \in \mathcal M(n)$, denote by $\Sigma_f$ the set of places,
    \begin{equation*}
        \Sigma_f := \Sigma_E \cup \{v \in \mathcal{P}: v\mid f_* \}.
    \end{equation*}
Note that if $f \in F_{(n,N),(w,w')}^{(\lambda,\eta)}$, then $\# \Sigma_f = \# \Sigma_E + (w-w')$.
\end{definition}

\begin{definition}[{\cite[Definition 5.10]{Park2025}}]
    Given a polynomial $f \in \mathcal{M}(\lambda, \eta ; Q,A)$ (see Definition \ref{defn:M(lambda,eta)}), we use the shorthand notation $\Omega_{\overline{f}^*}$ for the set of finite Cartesian products of local characters,
    \begin{align}
    \begin{split}
        \Omega_1 &= \prod_{v \in \Sigma_f} \text{Hom}(\text{Gal}(\overline{K}_v/K_v),\mu_\ell), \\
        \Omega_{\overline{f}^*} &= \prod_{v \in \Sigma_f} \text{Hom}(\text{Gal}(\overline{K}_v/K_v),\mu_\ell) \times \prod_{\substack{v \mid f^* \\ v \nmid f_a}} \text{Hom}_{ram}(\text{Gal}(\overline{K}_v/K_v),\mu_\ell),
    \end{split}
    \end{align}
    such that the component $\chi_v$ is ramified if $v \mid f^*$, and we ignore the local characters at any places $v$ dividing the auxiliary factor $f_a$ of $f$. In particular, we enlarge the set $\Sigma$ from Definition \ref{definition:basic_local_def} to include places $v \mid f_*$ and set $\Sigma = \Sigma_f$, even though $\chi_{f,v}$ is ramified at such places.
\end{definition}

\begin{proposition}[c.f.{\cite[Proposition 5.13]{Park2025}}] \label{prop:key-prop}
    Assume the notations and conditions as stated in Definition \ref{defn:important_maps}. Let $E/K$ be an elliptic curve satisfying Condition \ref{condition}. Given $f \in (\phi_{d^*,j^*,k^*,R^*} \circ \phi_{d_a})^{-1}(h)$, define the characters 
    \begin{equation*}
        \psi_f := (\chi_{f,v})_{v \in \Sigma_f(\overline{h}^*)} \in \Omega_{\overline{h}^*}, \qquad \psi'_f := (\chi_{f,v})_{v \in \Sigma_f(\overline{f}^*)} \in \Omega_{\overline{f}^*}.
    \end{equation*}
    Let $\delta_h: \mathbb{Z}_{\geq 0} \to [0,1]$ denote the probability distribution
    \begin{equation*}
        \delta_h(J) := \frac{\#\{f \in (\phi_{d^*,j^*,k^*,R^*} \circ \phi_{d_a})^{-1}(h) : \mathrm{rk}(\psi_f) = J\}}{\# (\phi_{d^*,j^*,k^*,R^*} \circ \phi_{d_a})^{-1}(h)}.
    \end{equation*}
    Denote by $\tilde{k} := \lambda_{d^*,j^*,k^*,R^*} \cdot k^*$. Then for any $n$ such that $n > \exp(\max\{\deg \Delta_E, 3 \log \ell\})$ and $\omega(h) \leq 2 \log n$, there exists a constant $B_{E,Q,q,\ell}$ depending on $E,Q,q,\ell$ such that
    \begin{equation*}
        \left|\frac{\#\{f \in (\phi_{d^*,j^*,k^*,R^*} \circ \phi_{d_a})^{-1}(h) : \mathrm{rk}(\psi'_f) = J\}}{\# (\phi_{d^*,j^*,k^*,R^*} \circ \phi_{d_a})^{-1}(h)} - (M_L^{\tilde{k}} \delta_h)(J) \right| < \lambda_{d^*,j^*,k^*,R^*} \cdot B_{E,Q,q,\ell} \cdot n^{-2 \log n + 6 \log \ell}.
    \end{equation*}
\end{proposition}
\begin{proof}
    The proof adapts that of \cite[Proposition 5.13]{Park2025}, but we explain what adjustments need to be made. 

    For the sake of notation, we abbreviate $\overline{\lambda} := \lambda_{d^*,j^*,k^*,R^*}$. We recall from Definition \ref{defn:important_maps} that there is a set bijection
    \begin{align*}
        (\phi_{d^*,j^*,k^*,R^*} \circ \phi_{d_a})^{-1}(h) &\cong \mathrm{Conf}_{\overline{\lambda}}(\mathcal{P}_{k^*}(d^*; Q, R^*)) \times \phi_{d_a}^{-1}(h) \\
        &\cong \frac{\mathrm{PConf}_{\overline{\lambda}}(\mathcal{P}_{k^*}(d^*; Q, R^*))}{S_{\overline{\lambda}}} \times \phi_{d_a}^{-1}(h).
    \end{align*}
    We have a natural surjection of degree $\# S_{\overline{\lambda}}$ given by
    \begin{equation*}
        \mathrm{PConf}_{\overline{\lambda}}(\mathcal{P}_{k^*}(d^*; Q, R^*)) \times \phi_{d_a}^{-1}(h) \to (\phi_{d^*,j^*,k^*,R^*} \circ \phi_{d_a})^{-1}(h),
    \end{equation*}
    which is obtained by forgetting the enumeration of irreducible factors of polynomials represented by an element in $\mathrm{Conf}_{\overline{\lambda}}(\mathcal{P}_{k^*}(d^*; Q, R^*))$. Hence, given a polynomial $f \in \phi_{d_a}^{-1}(h) \to (\phi_{d^*,j^*,k^*,R^*} \circ \phi_{d_a})^{-1}(h)$, all the fibers of $f$ with respect to the above surjection have identical Selmer structures $\mathrm{Sel}(E[\ell], \psi'_f)$. This allows us to rewrite the expression for any two non-negative integers $J_0, J$ (where we will use $J_{\overline{\lambda}}$ in place of $J$ for clarification)
    \begin{equation} \label{eqn:key-quantity-1}
        \#\{f \in (\phi_{d^*,j^*,k^*,R^*} \circ \phi_{d_a})^{-1}(h) : \mathrm{rk}(\psi'_f) = J_{\overline{\lambda}}, \mathrm{rk}(\psi_f) = J_0\}
    \end{equation}
    as 
    \begin{align*}
        \# S_{\overline{\lambda}} \cdot (\ref{eqn:key-quantity-1}) = \#\{(g, f_a) \in \mathrm{PConf}_{\overline{\lambda}}(\mathcal{P}_{k^*}(d^*; Q, R^*)) \times \phi_{d_a}^{-1}(h) : \mathrm{rk}(\psi'_f) = J_{\overline{\lambda}}, \mathrm{rk}(\psi_f) = J_0\}.
    \end{align*}
    Given a fixed choice of a Cartesian product of characters $\psi \in \Omega_{\overline{h}^*}$, we consider the expression
    \begin{equation} \label{eqn:key-quantity-2}
        \#\{(g, f_a) \in \mathrm{PConf}_{\overline{\lambda}}(\mathcal{P}_{k^*}(d^*; Q, R^*)) \times \phi_{d_a}^{-1}(h) : \psi_f = \psi, \; \mathrm{rk}(\psi'_f) = J_{\overline{\lambda}}, \mathrm{rk}(\psi) = J_0\}.
    \end{equation}
    Observe that
    \begin{equation*}
        \# S_{\overline{\lambda}} \cdot (\ref{eqn:key-quantity-1}) = \sum_{\psi \in \Omega_{\overline{h}^*}} (\ref{eqn:key-quantity-2}).
    \end{equation*}
    Hence
    % , we can reformulate the following expression appearing in the inequality appearing in the statement of the proposition as
    \begin{align} \label{eqn:double_summation}
    \begin{split}
    & \hspace{14pt} \frac{\#\{f \in (\phi_{d^*,j^*,k^*,R^*} \circ \phi_{d_a})^{-1}(h) : \mathrm{rk}(\psi'_f) = J\}}{\# (\phi_{d^*,j^*,k^*,R^*} \circ \phi_{d_a})^{-1}(h)} \\
    &= \frac{\sum_{J_0 = -\infty}^\infty (\ref{eqn:key-quantity-1})}{\# (\phi_{d^*,j^*,k^*,R^*} \circ \phi_{d_a})^{-1}(h)} = \frac{1}{\# S_{\overline{\lambda}}} \cdot \frac{\sum_{J_0 = -\infty}^\infty \sum_{\psi \in \Omega_{\overline{h}^*}} (\ref{eqn:key-quantity-2})}{\# (\phi_{d^*,j^*,k^*,R^*} \circ \phi_{d_a})^{-1}(h)} \\
    &= \sum_{J_0 = -\infty}^\infty \sum_{\psi \in \Omega_{\overline{h}^*}} \frac{\# \{(g,f_a) \in  \mathrm{PConf}_{\overline{\lambda}}(\mathcal{P}_{k^*}(d^*; Q, R^*)) \times \phi_{d_a}^{-1}(h) : \psi_f = \psi\}}{\# \mathrm{PConf}_{\overline{\lambda}}(\mathcal{P}_{k^*}(d^*; Q, R^*)) \times \phi_{d_a}^{-1}(h)} \\
    & \hspace{60pt} \times \frac{(\ref{eqn:key-quantity-2})}{\# \{(g,f_a) \in  \mathrm{PConf}_{\overline{\lambda}}(\mathcal{P}_{k^*}(d^*; Q, R^*)) \times \phi_{d_a}^{-1}(h) : \psi_f = \psi\}}.
    \end{split}
    \end{align}
    
    Now we simplify the first fraction appearing in the double sum \eqref{eqn:double_summation}. Let $h_1, h_2, \cdots, h_\omega$ be distinct irreducible factors of $h$, and write 
    \begin{equation*}
        \hat{\omega} := \begin{cases}
            \omega(h) &\text{ if } \ell \geq 5 \text{ or } K_Q(\sqrt[\ell]{h_1}, \cdots, \sqrt[\ell]{h_\omega}) \cap K(E[\ell]) = K, \\
            \omega(h) - 1 &\text{ if } \ell \leq 3 \text{ and } K_Q(\sqrt[\ell]{h_1}, \cdots, \sqrt[\ell]{h_\omega}) \cap K(E[\ell]) \neq K.
        \end{cases}
    \end{equation*}
    Let us denote by $\widetilde{\Omega_{\overline{h}^*}}$ the following set of Cartesian products of local characters:
    \begin{equation*}
        \widetilde{\Omega_{\overline{h}^*}} := \left\{ \psi \in \Omega_{\overline{h}^*} : \exists (g, f_a) \in \mathrm{PConf}_{\overline{\lambda}}(\mathcal{P}_{k^*}(d^*; Q, R^*)) \times \phi_{d_a}^{-1}(h) \text{ s.t. } \psi_f = \psi \right\}.
    \end{equation*}
    Proposition \ref{proposition:equidistribution} shows that 
    \begin{equation*}
        \# \widetilde{\Omega_{\overline{h}^*}} = \begin{cases}
            \# \Omega_{\overline{h}^*} &\text{ if } \ell \geq 5, \text{ or } K_Q(\sqrt[\ell]{h_1}, \cdots, \sqrt[\ell]{h_{\omega(h)}}) \cap K(E[\ell]) = K, \\
            (1/\ell) \cdot \# \Omega_{\overline{h}^*} &\text{ if } \ell \leq 3, \text{ and } K_Q(\sqrt[\ell]{h_1}, \cdots, \sqrt[\ell]{h_{\omega(h)}}) \cap K(E[\ell]) \neq K.
        \end{cases}
    \end{equation*}
    Then by Proposition \ref{proposition:equidistribution}, for any $\psi \in \widetilde{\Omega_{\overline{h}^*}}$ we have:
    \begin{equation*}
        \left| \frac{\# \{(g,f_a) \in  \mathrm{PConf}_{\overline{\lambda}}(\mathcal{P}_{k^*}(d^*; Q, R^*)) \times \phi_{d_a}^{-1}(h) : \psi_f = \psi\}}{\# \mathrm{PConf}_{\overline{\lambda}}(\mathcal{P}_{k^*}(d^*; Q, R^*)) \times \phi_{d_a}^{-1}(h)} - \frac{1}{\# \widetilde{\Omega_{\overline{h}^*}}} \right| < \hat{C}_{E,Q,q,\ell} \cdot n^{-2 \log n + 2 \log \ell}.
    \end{equation*}
    Using the trivial bound that (\ref{eqn:key-quantity-2}) $\leq 1$ and the fact that
    \begin{equation*}
         \# \widetilde{\Omega_{\overline{h}^*}} \leq \# \Omega_{\overline{h}^*} = (\ell(\ell-1))^{\# \Sigma_E + \omega(h)} \leq \ell^{2 \# \Sigma_E} \cdot n^{4 \log \ell},
    \end{equation*}
    we obtain
    \begin{align} \label{eqn:double_summation:rewrite1}
    \begin{split}
        & \hspace{15pt} \left|(\ref{eqn:double_summation}) - \sum_{J_0 = -\infty}^\infty \sum_{\psi \in \widetilde{\Omega_{\overline{h}^*}}} \frac{1}{\# \widetilde{\Omega_{\overline{h}^*}}} \cdot \frac{(\ref{eqn:key-quantity-2})}{\# \{(g,f_a) \in  \mathrm{PConf}_{\overline{\lambda}}(\mathcal{P}_{k^*}(d^*; Q, R^*)) \times \phi_{d_a}^{-1}(h) : \psi_f = \psi\}} \right| \\
        &\leq \ell^{2 \# \Sigma_E} \cdot \hat{C}_{E,Q,q,\ell} \cdot n^{-2 \log n + 6 \log \ell}.
    \end{split}
    \end{align}
    
    Next, we deal with the second fraction appearing in the double sum (\ref{eqn:double_summation}). Writing $\overline{g}$ for a fixed choice of an element in $\mathrm{PConf}_{\overline{\lambda}}(\mathcal{P}_{k^*}(d^*; Q, R^*))$, we see that the second fraction is (given fixed choices of $\psi \in \widetilde{\Omega_{\overline{h}^*}}$)
    \begin{align*}
        & \hspace{14pt} \frac{(\ref{eqn:key-quantity-2})}{\# \{(g,f_a) \in  \mathrm{PConf}_{\overline{\lambda}}(\mathcal{P}_{k^*}(d^*; Q, R^*)) \times \phi_{d_a}^{-1}(h) : \psi_f = \psi\}} \\
        &= \sum_{\overline{g} \in \mathrm{PConf}_{\overline{\lambda}}(\mathcal{P}_{k^*}(d^*; Q, R^*))}\frac{\# \{(\overline{g},f_a) \in  \mathrm{PConf}_{\overline{\lambda}}(\mathcal{P}_{k^*}(d^*; Q, R^*)) \times \phi_{d_a}^{-1}(h) : \psi_f = \psi\}}{\# \{(g,f_a) \in  \mathrm{PConf}_{\overline{\lambda}}(\mathcal{P}_{k^*}(d^*; Q, R^*)) \times \phi_{d_a}^{-1}(h) : \psi_f = \psi\}} \\
        & \hspace{50pt} \times \frac{\#\{(\overline{g}, f_a) \in \mathrm{PConf}_{\overline{\lambda}}(\mathcal{P}_{k^*}(d^*; Q, R^*)) \times \phi_{d_a}^{-1}(h) : \psi_f = \psi, \; \mathrm{rk}(\psi'_f) = J_{\overline{\lambda}}, \mathrm{rk}(\psi) = J_0\}}{\# \{(\overline{g},f_a) \in  \mathrm{PConf}_{\overline{\lambda}}(\mathcal{P}_{k^*}(d^*; Q, R^*)) \times \phi_{d_a}^{-1}(h) : \psi_f = \psi\}}.
    \end{align*}
    To handle the second factor in the sum in the last expression, we need to rewrite the difference between $\mathrm{rk}(\psi'_f)$ and $\mathrm{rk}(\psi)$ for a fixed choice of $\overline{g} \in \mathrm{PConf}_{\overline{\lambda}}(\mathcal{P}_{k^*}(d^*; Q, R^*))$ using Cartesian products of local characters. Because elements in $\phi_{d_a}^{-1}(h)$ are primes in $\mathcal{P}_0$, two elements $(\overline{g}, f_{a,1})$ and $(\overline{g}, f_{a,2}) \in \mathrm{PConf}_{\overline{\lambda}}(\mathcal{P}_{k^*}(d^*; Q, R^*))$ with $\psi_{\overline{g} \cdot f_{a,1}} = \psi_{\overline{g} \cdot f_{a,2}})$ also satisfy $\mathrm{rk}(\psi'_{\overline{g} \cdot f_{a,1}}) = \mathrm{rk}(\psi'_{\overline{g} \cdot f_{a,2}})$ if $\psi'_{\overline{g} \cdot f_{a,1}} = \psi'_{\overline{g} \cdot f_{a,2}}$, both of which are elements of $\Omega_{\overline{g} \cdot \overline{h}^*}$. Hence 
    \begin{align*}
        & \hspace{14pt} \frac{(\ref{eqn:key-quantity-2})}{\# \{(g,f_a) \in  \mathrm{PConf}_{\overline{\lambda}}(\mathcal{P}_{k^*}(d^*; Q, R^*)) \times \phi_{d_a}^{-1}(h) : \psi_f = \psi\}} \\
        &= \sum_{\overline{g} \in \mathrm{PConf}_{\overline{\lambda}}(\mathcal{P}_{k^*}(d^*; Q, R^*))} \frac{\# \{(\overline{g},f_a) \in  \mathrm{PConf}_{\overline{\lambda}}(\mathcal{P}_{k^*}(d^*; Q, R^*)) \times \phi_{d_a}^{-1}(h) : \psi_f = \psi\}}{\# \{(g,f_a) \in  \mathrm{PConf}_{\overline{\lambda}}(\mathcal{P}_{k^*}(d^*; Q, R^*)) \times \phi_{d_a}^{-1}(h) : \psi_f = \psi\}} \\
        & \hspace{15pt} \times \sum_{\psi' \in \Omega_{\overline{g} \cdot \overline{h}^*}} \frac{\#\{(\overline{g}, f_a) \in \mathrm{PConf}_{\overline{\lambda}}(\mathcal{P}_{k^*}(d^*; Q, R^*)) \times \phi_{d_a}^{-1}(h) : \psi_f = \psi, \; \psi'_f = \psi', \; \mathrm{rk}(\psi') = J_{\overline{\lambda}}\}}{\# \{(\overline{g},f_a) \in  \mathrm{PConf}_{\overline{\lambda}}(\mathcal{P}_{k^*}(d^*; Q, R^*)) \times \phi_{d_a}^{-1}(h) : \psi_f = \psi\}}.
    \end{align*}
    We apply Proposition \ref{proposition:equidistribution} once again to the ratios in the two sums. The first satisfies
    \begin{align*}
        & \left| \frac{\# \{(\overline{g},f_a) \in  \mathrm{PConf}_{\overline{\lambda}}(\mathcal{P}_{k^*}(d^*; Q, R^*)) \times \phi_{d_a}^{-1}(h) : \psi_f = \psi\}}{\# \{(g,f_a) \in  \mathrm{PConf}_{\overline{\lambda}}(\mathcal{P}_{k^*}(d^*; Q, R^*)) \times \phi_{d_a}^{-1}(h) : \psi_f = \psi\}} - \frac{1}{\# \mathrm{PConf}_{\overline{\lambda}}(\mathcal{P}_{k^*}(d^*; Q, R^*))} \right| \\
        & \leq 2 \cdot \hat{C}_{E,Q,q,\ell} \cdot n^{-2 \log n + 2 \log \ell},
    \end{align*}
    while the second satisfies
    \begin{align*}
        \biggl| \sum_{\psi' \in \Omega_{\overline{g} \cdot \overline{h}^*}} & \frac{\#\{(\overline{g}, f_a) \in \mathrm{PConf}_{\overline{\lambda}}(\mathcal{P}_{k^*}(d^*; Q, R^*)) \times \phi_{d_a}^{-1}(h) : \psi_f = \psi, \; \psi'_f = \psi', \; \mathrm{rk}(\psi'_f) = J_{\overline{\lambda}}\}}{\# \{(\overline{g},f_a) \in  \mathrm{PConf}_{\overline{\lambda}}(\mathcal{P}_{k^*}(d^*; Q, R^*)) \times \phi_{d_a}^{-1}(h) : \psi_f = \psi\}} \\
        &- \frac{\# \{\psi' \in \Omega_{\overline{g} \cdot \overline{h}^*} : (\psi'_v)_{v \in \Sigma_f(\overline{h}^*)} = \psi, \; \mathrm{rk}(\psi') = J_{\overline{\lambda}}\}}{\# \{\psi' \in \Omega_{\overline{g} \cdot \overline{h}^*} : (\psi'_v)_{v \in \Sigma_f(\overline{h}^*)} = \psi\}} \biggr| \leq 2 \cdot \hat{C}_{E,Q,q,\ell} \cdot n^{-2 \log n + 2 \log \ell}.
    \end{align*}
    Combining both expressions, we obtain
    \begin{align*}
            & \biggl| \frac{(\ref{eqn:key-quantity-2})}{\# \{(g,f_a) \in  \mathrm{PConf}_{\overline{\lambda}}(\mathcal{P}_{k^*}(d^*; Q, R^*)) \times \phi_{d_a}^{-1}(h) : \psi_f = \psi\}} \\
            & \; - \frac{1}{\# \mathrm{PConf}_{\overline{\lambda}}(\mathcal{P}_{k^*}(d^*; Q, R^*))} \cdot \sum_{\overline{g} \in \mathrm{PConf}_{\overline{\lambda}}(\mathcal{P}_{k^*}(d^*; Q, R^*))}  \frac{\# \{\psi' \in \Omega_{\overline{g} \cdot \overline{h}^*} : (\psi'_v)_{v \in \Sigma_f(\overline{h}^*)} = \psi, \; \mathrm{rk}(\psi') = J_{\overline{\lambda}}\}}{\# \{\psi' \in \Omega_{\overline{g} \cdot \overline{h}^*} : (\psi'_v)_{v \in \Sigma_f(\overline{h}^*)} = \psi\}} \biggr| \\
            &\leq 8 \cdot \hat{C}_{E,Q,q,\ell}^2 \cdot n^{-2 \log n + 2 \log \ell}.
    \end{align*}
    Notice that the expression
    \begin{align*}
        \frac{\# \{\psi' \in \Omega_{\overline{g} \cdot \overline{h}^*} : (\psi'_v)_{v \in \Sigma_f(\overline{h}^*)} = \psi, \; \mathrm{rk}(\psi') = J_{\overline{\lambda}}\}}{\# \{\psi' \in \Omega_{\overline{g} \cdot \overline{h}^*} : (\psi'_v)_{v \in \Sigma_f(\overline{h}^*)} = \psi\}} = \frac{\# \{\psi' \in \Omega_{\overline{g} \cdot \overline{h}^*} : (\psi'_v)_{v \in \Sigma_f(\overline{h}^*)} = \psi, \; \mathrm{rk}(\psi') = J_{\overline{\lambda}}\}}{(\ell(\ell-1))^{\overline{\lambda}}}
    \end{align*}
    can be computed applying Proposition \ref{prop:changes_local_rank} by $\overline{\lambda}$ many times. To elaborate, denote by $\delta_\psi: \mathbb{Z}_{\geq 0} \to [0,1]$ the probability distribution defined as
    \begin{equation*}
        \delta_\psi(J) := \begin{cases}
            1 &\text{ if } J = \mathrm{rk}(\psi), \\
            0 &\text{ otherwise}.
        \end{cases}
    \end{equation*}
    Then we have
    \begin{align*}
        & \biggl| \sum_{\overline{g} \in \mathrm{PConf_{\overline{\lambda}}}(\mathcal{P}_{k^*}(d^*; Q, R^*))}\frac{\# \{\psi' \in \Omega_{\overline{g} \cdot \overline{h}^*} : (\psi'_v)_{v \in \Sigma_f(\overline{h}^*)} = \psi, \; \mathrm{rk}(\psi') = J_{\overline{\lambda}}\}}{\# \mathrm{PConf}_{\overline{\lambda}}(\mathcal{P}_{k^*}(d^*; Q, R^*)) \cdot \# \{\psi' \in \Omega_{\overline{g} \cdot \overline{h}^*} : (\psi'_v)_{v \in \Sigma_f(\overline{h}^*)} = \psi\}} \\
        & \hspace{15pt} - (M_L^{\tilde{k}} \delta_\psi)(J_{\overline{\lambda}}) \biggr| < 2 \cdot C_{E,Q,q,\ell} \cdot \ell^{3 \# \Sigma_E(\sigma)} \cdot n^{-2 \log n},
    \end{align*}
    where $C_{E,Q,q,\ell}$ is the constant appearing in Proposition \ref{prop:tchiv_density} (A detailed proof of this can be found in \cite[Proof of Proposition 5.13, pp. 3306--3308]{Park2025}). We hence have
    \begin{align*}
            & \biggl| \frac{(\ref{eqn:key-quantity-2})}{\# \{(g,f_a) \in  \mathrm{PConf}_{\overline{\lambda}}(\mathcal{P}_{k^*}(d^*; Q, R^*)) \times \phi_{d_a}^{-1}(h) : \psi_f = \psi\}} - (M_L^{\tilde{k}} \delta_\psi)(J_{\overline{\lambda}}) \biggr| \\
            & \hspace{15pt} \leq (2 \cdot C_{E,Q,q,\ell} \cdot \ell^{3 \# \Sigma_E(\sigma)} + 8 \cdot \hat{C}_{E,Q,q,\ell}^2) \cdot n^{-2 \log n + 2 \log \ell}.
    \end{align*}
    Combining this with (\ref{eqn:double_summation:rewrite1}) yields
    \begin{align*}
        \left| (\ref{eqn:double_summation}) - \sum_{J_0 = -\infty}^\infty \sum_{\psi \in \widetilde{\Omega_{\overline{h}^*}}} \frac{1}{\# \widetilde{\Omega_{\overline{h}^*}}} \cdot (M_L^{\tilde{k}} \delta_\psi)(J_{\overline{\lambda}}) \right| \leq (2C_{E,Q,q,\ell} + 8\hat{C}_{E,Q,q,\ell}^2)\cdot\ell^{3 \# \Sigma_E} n^{-2 \log n + 6 \log \ell}.
    \end{align*}
    By Proposition \ref{proposition:equidistribution}, for any $J_0 \geq 0$ we have
    \begin{equation*}
        \left| \delta_h(J_0) - \sum_{J_0 = -\infty}^\infty \sum_{\psi \in \widetilde{\Omega_{\overline{h}^*}}} \frac{1}{\# \widetilde{\Omega_{\overline{h}^*}}} \cdot \delta_\psi(J_0) \right| < \hat{C}_{E,Q,q,\ell} \cdot n^{-2 \log n + 2 \log \ell}.
    \end{equation*}
    Since there are $2 \tilde{k} + 1$ possible values of $J_0$ which can reach a given state $J_{\overline{\lambda}}$ after the operator $M_L^{\tilde{k}}$ is applied, we obtain
    \begin{equation*}
        \left| (M_L^{\tilde{k}} \delta_h)(J_{\overline{\lambda}}) - \sum_{J_0 = -\infty}^\infty \sum_{\psi \in \widetilde{\Omega_{\overline{h}^*}}} \frac{1}{\# \widetilde{\Omega_{\overline{h}^*}}} \cdot (M_L^{\tilde{k}} \delta_\psi)(J_{\overline{\lambda}}) \right| < (2 \tilde{k} + 1) \cdot \hat{C}_{E,Q,q,\ell} \cdot n^{-2 \log n + 2 \log \ell}.
    \end{equation*}
    Using the fact that $2\tilde{k} + 1 \leq 5 \lambda_{d^*,j^*,k^*,R^*}$ whenever $\lambda_{d^*,j^*,k^*,R^*} \geq 1$ and $k^* \leq 2$, we get
    \begin{align*}
        \left| (\ref{eqn:double_summation}) - (M_L^{\tilde{k}} \delta_h)(J_{\overline{\lambda}}) \right| \leq \lambda_{d^*,j^*,k^*,R^*} \cdot (10 \cdot C_{E,Q,q,\ell} + 40 \cdot \hat{C}_{E,Q,q,\ell}^2) \cdot \ell^{3 \# \Sigma_E} \cdot n^{-2 \log n + 6 \log \ell}.
    \end{align*}
    Taking 
    \[B_{E,Q,q,\ell} = (10 \cdot C_{E,Q,q,\ell} + 40 \cdot \hat{C}_{E,Q,q,\ell}^2) \cdot \ell^{3 \# \Sigma_E}\] 
    finishes the proof.
\end{proof}

\subsection{Global Selmer structures}

In the previous subsection, we saw that Markov models can still be utilized to understand variations in local Selmer structures decorated with conditions pertaining to irreducible polynomials in arithmetic progressions. We proceed to prove the following result on the distribution of global Selmer groups $\mathrm{Sel}_{1-\sigma_f}(A_f/K)$ as $f$ varies over $\mathcal{M}(n, w : N, w'; Q,A) $.

\begin{proposition}[c.f. {\cite[Proposition 6.11]{Park2025}}] \label{prop:global_selmer}
    Let $n > N$, $w' < w < 2 \log n$ be four positive integers, such that $w' = (1-\epsilon)w$ for some small enough $0 < \epsilon < 1$. Fix a monic polynomial $Q$ and a polynomial $A$ of degree less than $\deg Q$. Suppose $n$ satisfies $$n > \max \left\{\exp(\exp(\exp(e))), 6 (\ell^3 + g_{E[\ell]}), \exp(\deg \Delta_E) \right\}.$$ Then there exists a constant $\tilde{B}_{E,Q,q,\ell}$ depending on $E, Q, q, \ell$ and a constant $\gamma_\ell \in (0,1)$ (from Lemma \ref{corollary:uniform_markov}) such that
    \begin{align*}
        & \hspace{15pt} \biggl| \frac{\# \{f \in \mathcal{M}(n, w : N, w'; Q,A) : \dim_{\mathbb{F}_\ell} \mathrm{Sel}_{1-\sigma_f}(E^{\chi_f}/K) = J\}}{\# \mathcal{M}(n, w : N, w'; Q,A)} - \rho(J) \biggr| \\
        &< \tilde{B}_{E,Q,q,\ell} \cdot \left(n^{-2 \log n + 6 \log \ell + 1} + \gamma_\ell^{w'} \cdot n^{4 \epsilon \log \ell} \right).
    \end{align*}
\end{proposition}
\begin{proof}
The proof adapts that of \cite[Proposition 6.11]{Park2025} to our setup, but the subset of primes $\mathcal{P}_k(i)$ gets replaced by $\mathcal{P}_k(i;Q,R)$ for any residues $R$ mod $Q$. The overall strategy is to extend the statistics of local Selmer structures obtained from Proposition \ref{prop:changes_local_rank} by enlarging the domain of global characters in the following fashion:
\begin{enumerate}
    \item Start with statistics of local Selmer structures over a \textit{subset} of elements $f \in M(\lambda, \eta; Q, A)$, given fixed choices of splitting partitions $\lambda \in \Lambda_{N,w',Q}^{\emph{la}}$, $\eta \in \Lambda_{n-N,w-w',Q}^{\emph{for}}$, and a fixed choice of $f_*$ (which we denote by $h_*$). 
    \item Obtain statistics of local Selmer structures over $\bigsqcup_{\lambda} \{f \in M(\lambda, \eta; Q, A) : f_* = h_*\}$, as $\lambda$ varies over $\Lambda_{N,w',Q}^{\emph{la}}$.
    \item Obtain statistics of local Selmer structures over $\bigsqcup_{\lambda} M(\lambda, \eta; Q, A)$ by varying choices of $h_*$.
    \item Obtain statistics of local Selmer structures over $M(n,w : N,w'; Q, A)$, which is identical to $\bigsqcup_{\lambda, \eta} M(\lambda, \eta; Q, A)$ with $\eta$ varying over $\Lambda_{n-N,w-w',Q}^{\emph{for}}$.
\end{enumerate}

Given a global character $\chi_f \in \mathrm{Hom}(\mathrm{Gal}(\overline{K}/K), \mathbb{Z}/\ell \mathbb{Z})$ such that the fixed field of its kernel is $K(\sqrt[\ell]{f})$, and a place $v$, we recall from Remark \ref{remark:local-global-selmer} that
\begin{equation*}
    \mathrm{Sel}_{1-\sigma_f}(E^{\chi_f}/K) = \mathrm{Sel}(E[\ell], (\chi_{f,v})_v).
\end{equation*}
We shall obtain the desired statistics on dimensions of $\mathrm{Sel}_{1-\sigma_f}(E^{\chi_f}/K)$ from statistics on dimensions of local Selmer structures $\mathrm{Sel}(E[\ell], (\chi_{f,v})_v)$ that we obtained from the previous subsection.

Taking four integers $n, N, w, w'$ as in the statement of the proposition, we fix splitting partitions $\lambda \in \Lambda_{N,w',Q}^{\emph{la}}$ and $\eta \in \Lambda_{n-N,w-w',Q}^{\emph{for}}$. Without loss of generality, assume that the subset of polynomials $M(\lambda, \eta; Q, A)$ is nonempty. By Lemma \ref{lemma:characterization_M(lambda,eta)}, there exists a surjection
\begin{equation*}
    \Phi_{\lambda, \eta, Q, A}: M(\lambda, \eta; Q, A) \to \prod_{\hat{i}, \hat{j}, \hat{k}, \hat{R}} \mathrm{Conf}_{\eta_{\hat{i},\hat{j},\hat{k},\hat{R}}}(\mathcal{P}_{\hat{k}}(\hat{i}; Q,\hat{R})),
\end{equation*}
which forgets all the irreducible factors lying in $\prod_{i,j,k,R} \mathrm{Conf}_{\lambda_{i,j,k,R}}(\mathcal{P}_k(i; Q,R))$. By definition, we have
\begin{equation*}
    \Phi_{\lambda, \eta, Q, A} = \left(\underset{\lambda_{d^*,j^*,k^*,R^*}}{\circ} \phi_{d^*,j^*,k^*,R^*} \right) \circ \phi_{d_a}.
\end{equation*}
In other words, $\Phi_{\lambda, \eta, Q, A}$ can be decomposed as the map $\phi_{d_a}$ followed by iterated compositions of maps $\phi_{d^*,j^*,k^*,R^*}$, as one vary over the set of all possible tuples $(d^*,j^*,k^*,R^*)$ such that $\lambda_{d^*,j^*,k^*,R^*} \neq 0$. 

\textbf{Step (1):} Given a fixed element $h_* \in \prod_{\hat{i}, \hat{j}, \hat{k}, \hat{R}} \mathrm{Conf}_{\eta_{\hat{i},\hat{j},\hat{k},\hat{R}}}(\mathcal{P}_{\hat{k}}(\hat{i}; Q,\hat{R}))$, consider the following Cartesian products of local characters, associated probability distribution $\delta_{h_*}: \mathbb{Z}_{\geq 0} \to [0,1]$, and quantity $d_\lambda$:
\begin{align}
    \begin{split}
        \Omega_{\overline{h_*}} &:= \prod_{v \in \Sigma_E} \mathrm{Hom}(\mathrm{Gal}(\overline{K}_v/K_v), \mu_\ell) \times \prod_{\substack{v \mid h_* \\ v \not\in \Sigma_E}} \mathrm{Hom}_{ram}(\mathrm{Gal}(\overline{K}_v/K_v), \mu_\ell), \\
        \delta_{h_*}(z) &:= \frac{\# \{\omega \in \Omega_{\overline{h_*}} : \mathrm{rk}(\omega) = z\}}{\# \Omega_{\overline{h_*}}}, \\
        d_\lambda &:= \sum_{d^*,j^*,R^*} (\lambda_{d^*,j^*,1,R^*} + 2 \cdot \lambda_{d^*,j^*,2,R^*}).
    \end{split}
\end{align}
Note that $w < 2 \log n$ implies $d_\lambda < 2 \log n$.
Given a non-negative integer $J$, consider the ratio
\begin{equation*}
    \frac{\# \{f \in \Phi_{\lambda, \eta, Q, A}^{-1}(h_*): \dim_{\mathbb{F}_\ell} \mathrm{Sel}_{1-\sigma_f}(E^{\chi_f}/K) = J\}}{\# \Phi_{\lambda, \eta, Q, A}^{-1}(h_*)}.
\end{equation*}
Applying Proposition \ref{prop:key-prop} iteratively over the collections of maps $\phi_{d^*,j^*,k^*,R^*}$ where $\lambda_{d^*,j^*,k^*,R^*} \neq 0$, and identifying $\mathrm{Sel}_{1-\sigma_f}(E^{\chi_f}/K)$ as a local Selmer structure, we obtain
\begin{align*}
    \begin{split}
        \biggl| \frac{\# \{f \in \Phi_{\lambda, \eta, Q, A}^{-1}(h_*): \dim_{\mathbb{F}_\ell} \mathrm{Sel}_{1-\sigma_f}(E^{\chi_f}/K) = J\}}{\# \Phi_{\lambda, \eta, Q, A}^{-1}(h_*)} - (M_L^{d_\lambda} \delta_{h_*})(J) \biggr| &< B_{E,Q,q,\ell} \cdot d_\lambda \cdot n^{-2 \log n + 6 \log \ell} \\ &< B_{E,Q,q,\ell} \cdot n^{-2 \log n + 6 \log \ell + 1},
    \end{split}
\end{align*}
where $B_{E,Q,q,\ell} > 0$ is the constant from Proposition \ref{prop:key-prop}.

%\textbf{[TODO: Swap order of Step(2) and Step(3), to allow room for generalization of results for number of fixed prime factors.]}

\textbf{Step (2):} We consider the statistics over the unions of subsets as we vary over $\lambda \in \Lambda_{N, w', Q}^{\emph{la}}$:
\begin{equation*}
    \frac{ \sum_{\lambda \in \Lambda_{N,w',Q}^{\emph{la}}}\# \{f \in \Phi_{\lambda, \eta, Q, A}^{-1}(h_*): \dim_{\mathbb{F}_\ell} \mathrm{Sel}_{1-\sigma_f}(E^{\chi_f}/K) = J\}}{\sum_{\lambda \in \Lambda_{N,w',Q}^{\emph{la}}}\# \Phi_{\lambda, \eta, Q, A}^{-1}(h_*)}.
\end{equation*}
Using the effective Chebotarev density theorem over the field extension $K(E[\ell])/K$, we obtain 
\begin{align*}
    \begin{split}
        & \hspace{15pt} \biggl| \frac{\sum_{\lambda \in \Lambda_{N,w',Q}^{\emph{la}}} \# \{f \in \Phi_{\lambda, \eta, Q, A}^{-1}(h_*): \dim_{\mathbb{F}_\ell} \mathrm{Sel}_{1-\sigma_f}(E^{\chi_f}/K) = J\}}{\sum_{\lambda \in \Lambda_{N,w',Q}^{\emph{la}}} \# \Phi_{\lambda, \eta, Q, A}^{-1}(h_*)} - (M^{w'-1} \delta_{h_*})(J) \biggr| \\
        &\leq \sum_{\lambda \in \Lambda_{N,w',Q}^{\emph{la}}} \frac{\# \Phi_{\lambda, \eta, Q, A}^{-1}(h_*)}{\sum_{\lambda \in \Lambda_{N,w',Q}^{\emph{la}}} \# \Phi_{\lambda, \eta, Q, A}^{-1}(h_*)} \cdot \biggl| \frac{\# \{f \in \Phi_{\lambda, \eta, Q, A}^{-1}(h_*): \dim_{\mathbb{F}_\ell} \mathrm{Sel}_{1-\sigma_f}(E^{\chi_f}/K) = J\}}{\# \Phi_{\lambda, \eta, Q, A}^{-1}(h_*)} \\
        & \hspace{200pt} - (M^{w'-1} \delta_{h_*})(J) \biggr| \\
        &< B_{E,Q,q,\ell} \cdot n^{-2 \log n + 6 \log \ell + 1}.
    \end{split}
\end{align*}
Here, $M$ is the Markov operator defined in Lemma \ref{corollary:uniform_markov}:
\begin{equation*}
    M := \left(1 - \frac{\ell}{\ell^2-1} \right) I + \frac{1}{\ell} M_L + \frac{1}{\ell^3-\ell} M_L^2.
\end{equation*}
As stated in \cite[p. 3315]{Park2025}, the Markov operator $M$ is applied $w'-1$ times, instead of $w'$ times, because we use one of the auxiliary places of $f$ to obtain equidistribution over Cartesian products of local characters over places $v \in \Sigma_f(\overline{f}^*)$.

\textbf{Step (3):} Now consider statistics as we vary over $h_* \in \prod_{\hat{i}, \hat{j}, \hat{k}, \hat{R}} \mathrm{Conf}_{\eta_{\hat{i},\hat{j},\hat{k},\hat{R}}}(\mathcal{P}_{\hat{k}}(\hat{i}; Q,\hat{R}))$. For this, we use geometric ergodicity of Markov operators with stationary distribution
\begin{equation*}
    \rho(z) := \prod_{k=0}^\infty \frac{1}{1+\ell^{-k}} \prod_{j=1}^z \frac{\ell}{\ell^j-1}.
\end{equation*}
from Lemma \ref{corollary:uniform_markov} to simplify $M^{w'-1} \delta_{h_*}$.
Since $w - w' \leq 2\epsilon \log n$, the proof of Proposition \ref{prop:changes_local_rank} (more precisely, \cite[Proposition 7.2]{KMR14} and \cite[Proposition 5.3]{Park2025}) gives the trivial bound
\begin{equation*}
    \mathbb{E}[\ell^{\delta_{h_*}}] \leq \ell^{\max_{\chi \in \Omega_E} \mathrm{rk}(\chi)} \cdot n^{4 \epsilon \log \ell} \leq B_{E,Q,q,\ell} \cdot n^{4 \epsilon \log \ell}.
\end{equation*}
Let $\gamma_\ell$ be the constant appearing in the conclusion of Lemma \ref{corollary:uniform_markov}. Then there exists some constant $c > 0$ such that
\begin{equation*}
    \sup_{z \in \mathbb{Z}_{\geq 0}} \biggl| (M^{w'-1} \delta_{h_*})(z) - \rho(z) \biggr| < 2c \cdot B_{E,Q,q,\ell} \cdot \gamma_\ell^{w'-1} \cdot n^{4 \epsilon \log \ell}.
\end{equation*}
Note that the right hand side of the above inequality is independent of $h_*$. Hence, we can use the above inequality to obtain
\begin{align*}
    \begin{split}
        & \hspace{15pt} \biggl| \frac{\sum_{\lambda \in \Lambda_{N,w',Q}^{\emph{la}}} \# \{f \in \mathcal{M}(\lambda, \eta; Q, A): \dim_{\mathbb{F}_\ell} \mathrm{Sel}_{1-\sigma_f}(E^{\chi_f}/K) = J\}}{\sum_{\lambda \in \Lambda_{N,w',Q}^{\emph{la}}} \# \mathcal{M}(\lambda, \eta; Q, A)} - \rho(J) \biggr| \\
        &\leq \sum_{h_*}\frac{\sum_{\lambda \in \Lambda_{N,w',Q}^{\emph{la}}} \Phi_{\lambda, \eta, Q, A}^{-1}(h_*)}{\sum_{\lambda \in \Lambda_{N,w',Q}^{\emph{la}}} \# \mathcal{M}(\lambda, \eta; Q, A)} \Bigg[ \biggl| \frac{\sum_{\lambda \in \Lambda_{N,w',Q}^{\emph{la}}} \# \{f \in \Phi_{\lambda, \eta, Q, A}^{-1}(h_*): \dim_{\mathbb{F}_\ell} \mathrm{Sel}_{1-\sigma_f}(E^{\chi_f}/K) = J\}}{\sum_{\lambda \in \Lambda_{N,w',Q}^{\emph{la}}} \# \Phi_{\lambda, \eta, Q, A}^{-1}(h_*)} \\
        &\hspace{200pt} - (M^{w'-1} \delta_{h_*})(J) \biggr| + \biggl| (M^{w'-1} \delta_{h_*})(J) - \rho(J) \biggr| \Bigg] \\
        &< B_{E,Q,q,\ell} \left( n^{-2 \log n + 6 \log \ell + 1} + 2c \cdot \gamma_\ell^{w'-1} \cdot n^{4 \epsilon \log \ell} \right) \\
        &< (B_{E,Q,q,\ell} + 2 c / \gamma_\ell)\left( n^{-2 \log n + 6 \log \ell + 1} + \gamma_\ell^{w'} \cdot n^{4 \epsilon \log \ell} \right).
    \end{split}
\end{align*}
Set $\tilde{B}_{E,Q,q,\ell} := B_{E,Q,q,\ell} + 2c / \gamma_\ell$.

\textbf{Step (4):} Lastly, we consider statistics as we range over $\eta \in \Lambda_{n-N,w-w',Q}^{\emph{for}}$. The relation
\begin{equation*}
    \mathcal{M}(n, w: N, w'; Q, A) = \bigsqcup_{\lambda \in \Lambda_{N,w',Q}^{\emph{la}}} \bigsqcup_{\eta \in \Lambda_{n-N,w-w',Q}^{\emph{for}}} \mathcal{M}(\lambda, \eta; Q, A)
\end{equation*}
implies
\begin{align*}
        & \hspace{15pt} \biggl| \frac{\# \{f \in \mathcal{M}(n, w : N, w'; Q,A) : \dim_{\mathbb{F}_\ell} \mathrm{Sel}_{1-\sigma_f}(E^{\chi_f}/K) = J\}}{\# \mathcal{M}(n, w : N, w'; Q,A)} - \rho(J) \biggr| \\
        &\leq \sum_{\eta \in \Lambda_{n-N,w-w',Q}^{\emph{for}}} \frac{\sum_{\lambda \in \Lambda_{N,w',Q}^{\emph{la}}} \# \mathcal{M}(\lambda, \eta; Q, A)}{\# \mathcal{M}(n, w : N, w'; Q, A)} \\
        & \hspace{30pt} \times \biggl| \frac{\sum_{\lambda \in \Lambda_{N,w',Q}^{\emph{la}}} \# \{f \in \mathcal{M}(\lambda, \eta; Q, A): \dim_{\mathbb{F}_\ell} \mathrm{Sel}_{1-\sigma_f}(E^{\chi_f}/K) = J\}}{\sum_{\lambda \in \Lambda_{N,w',Q}^{\emph{la}}} \# \mathcal{M}(\lambda, \eta; Q, A)} - \rho(J) \biggr| \\
        &< \tilde{B}_{E,Q,q,\ell}\left(n^{-2 \log n + 6 \log \ell + 1} + \gamma_\ell^{w'} \cdot n^{4 \epsilon \log \ell} \right),
\end{align*}
proving the proposition.
\end{proof}

\subsection{Proof of main results}
We are now ready to prove Theorem \ref{theorem:main2}.
\begin{proof}[Proof of Theorem \ref{theorem:main2}]
    Recall from Theorem \ref{theorem:decomposition} that
    \begin{align*}
    \begin{split}
        & \hspace{15pt} \#\mathcal{M}(n ; Q, A) - \sum_{w = \lfloor(1-B) \log n \rfloor}^{\lceil(1+B) \log n \rceil} \hspace{5pt} \sum_{w' = \lfloor(1 - \frac{1}{\log \log \log n}) w \rfloor}^{w} \hspace{5pt} \sum_{N = w' \mathfrak{n}}^n \#\mathcal{M}(n,w:N,w'; Q, A) \\
        &\ll_{Q,q} \frac{\# \mathcal{M}(n ; Q, A)}{B^2 \log(2n)}.
    \end{split}
    \end{align*}
    We take $B = \frac{1}{2}$. In context of notations appearing in Proposition \ref{prop:global_selmer}, we set $\epsilon := \frac{1}{\log \log \log n}$. Then
    % one obtains over any subsets of polynomials $\mathcal{M}(n,w:N,w'; Q, A)$ appearing in the summation above,
    \begin{align*}
        & \hspace{15pt} \biggl| \frac{\# \{f \in \mathcal{M}(n, w : N, w'; Q,A) : \dim_{\mathbb{F}_\ell} \mathrm{Sel}_{1-\sigma_f}(E^{\chi_f}/K) = J\}}{\# \mathcal{M}(n, w : N, w'; Q,A)} - \rho(J) \biggr| \\
        &< \tilde{B}_{E,Q,q,\ell} \left( n^{-2 \log n + 6 \log \ell + 1} + \gamma_\ell^{\frac{1}{2} \log n - \frac{1}{2} \frac{\log n}{\log \log \log n}} \cdot n^{4 \frac{\log \ell}{\log \log \log n}} \right) \\
        &\leq 2 \tilde{B}_{E,Q,q,\ell} \cdot n^{\frac{1}{2} \log \gamma_\ell - \frac{1}{2} \frac{\log \gamma_\ell}{\log \log \log n} + \frac{4 \log \ell}{\log \log \log n}}.
    \end{align*}
    Note that the leading exponent of $n$, namely $\frac{1}{2} \log \gamma_\ell$, is negative because $\gamma_\ell \in (0,1)$ as in Lemma \ref{corollary:uniform_markov}. Summing the above inequality over all $w, w', N$, we obtain for any non-negative $J \geq 0$ and sufficiently large $n$,
    \begin{equation*}
        \biggl| \frac{\# \{f \in \mathcal{M}(n; Q, A) : \dim_{\mathbb{F}_\ell} \mathrm{Sel}_{1-\sigma_f}(E^{\chi_f}/K) = J\}}{\# \mathcal{M}(n; Q, A)} - \rho(J) \biggr| \ll_{Q,q} \frac{1}{\log n}.
    \end{equation*}
\end{proof}
 
Theorem \ref{theorem:main1} now follows almost immediately from Theorem \ref{theorem:main2}.
\begin{proof}[Proof of Theorem \ref{theorem:main1}]
    By \cite[Propositions 2.1, 6.3]{MR07}, we have
    \begin{equation*}
        \mathrm{rank}(E/K(\sqrt[\ell]{f})) - \mathrm{rank}(E/K) \leq (\ell - 1) \dim_{\mathbb{F}_\ell} \mathrm{Sel}_{1-\sigma_f}(E^{\chi_f}/K).
    \end{equation*}
    In particular, if $\dim_{\mathbb{F}_\ell} \mathrm{Sel}_{1-\sigma_f}(E^{\chi_f}/K) = 0$, then $\mathrm{rank}(E/K(\sqrt[\ell]{f})) = \mathrm{rank}(E/K)$. The limit inferior (as $n \to \infty$) of the probability of the event $\mathrm{rank}(E/K(\sqrt[\ell]{f})) = \mathrm{rank}(E/K)$
    is bounded below by the asymptotic probability (as $n \to \infty$) that $\dim_{\mathbb{F}_\ell} \mathrm{Sel}_{1-\sigma_f}(E^{\chi_f}/K) = 0$. The latter probability is provided by Theorem \ref{theorem:main2}:
    \begin{equation*}
        \lim_{n \to \infty} \frac{\#\{f \in \mathcal{M}(n;Q,A) : \dim_{\mathbb{F}_\ell} \mathrm{Sel}_{1-\sigma_f}(E^{\chi_f}/K) = 0 \}}{\# \mathcal{M}(n;Q,A)} = \prod_{i=0}^\infty \frac{1}{1+\ell^{-i}}.
    \end{equation*}
    This completes the proof of Theorem \ref{theorem:main1}.
\end{proof}

\nocite{*}
\bibliographystyle{alpha}
\bibliography{references}

\end{document}

%% file: preamble.tex
\usepackage[margin=0.75in]{geometry}
\usepackage[hidelinks]{hyperref}
\usepackage[OT2, T1]{fontenc}
\usepackage[english]{babel}
\usepackage[utf8]{inputenc}
\usepackage{csquotes}
\usepackage[final]{microtype}
\usepackage{lmodern}
\usepackage{amsthm}
\usepackage{amssymb}
\usepackage{mathrsfs}
\usepackage{enumerate}
\usepackage{tikz-cd} % commutative diagrams
\usepackage{tikz}
\usepackage{multicol}
\usepackage{multirow}
\usepackage{tikz-qtree}
\usepackage{rotating}
\usepackage{graphicx}
\usepackage{comment}
\usepackage{enumitem}
\usepackage{manfnt}
\usetikzlibrary{arrows,calc,matrix,trees,arrows.meta,positioning,decorations.pathreplacing,bending}

\newtheorem{theorem}{Theorem}[section]

\newtheorem{proposition}[theorem]{Proposition}
\newtheorem{lemma}[theorem]{Lemma}

\theoremstyle{definition}

\newtheorem{remark}[theorem]{Remark}
\newtheorem{definition}[theorem]{Definition}

\newtheorem{condition}[theorem]{Condition}

\newcommand{\Z}{\mathbb{Z}}
\newcommand{\N}{\mathbb{N}} %Integers or Norm Map
\newcommand{\F}{\mathbb{F}}
\newcommand{\E}{\mathbb{E}}
\newcommand{\R}{\mathbb{R}}

\newcommand{\M}{\mathcal{M}}
\newcommand{\PP}{\mathcal{P}}

\newcommand{\Oh}{\mathcal{O}} %Notation for ring of integers
\newcommand{\Sel}{\mathrm{Sel}} %Selmer Groups
\newcommand{\Gal}{\mathrm{Gal}} %Galois Groups
\newcommand{\et}{\text{\'et}}

\newcommand{\genlegendre}[4]{%
  \genfrac{(}{)}{}{#1}{#3}{#4}%
  \if\relax\detokenize{#2}\relax\else_{\!#2}\fi
}

\DeclareSymbolFont{cyrletters}{OT2}{wncyr}{m}{n}
\DeclareMathSymbol{\Sha}{\mathalpha}{cyrletters}{"58}

%% file: references.bib
@article{Park2025,
      title={On the prime {S}elmer ranks of cyclic prime twist families of elliptic curves over global function fields}, 
      author={Park, {S.\,W.}},
      year={2025},
      journal = {Compositio Mathematica},
      pages = {3277-3320},
      volume = {161},
      number = {12}
}

@article{AP19,
    author = {A. Afshar and S. Porritt},
    title = {The function field {Sathe–Selberg} formula in arithmetic progressions and ‘short intervals’},
    journal = {Acta Arithmetica},
    volume = {187},
    number = {2},
    year = {2019},
    pages = {101--124}
}

@misc{Park2026,
    title = {Secondary terms for first moments of {Selmer} groups of twists of elliptic curves over global function fields},
    author = {Park, {S.\,W.}},
    year = {2026},
    eprint = {2606.14274},
    archivePrefix={arXiv},
    primaryClass={math.NT},
    note = {arXiv:2606.14274}
}

@book{Rosen02,
  author    = {M. Rosen},
  title     = {Number Theory in Function Fields},
  series    = {Graduate Texts in Mathematics},
  volume    = {210},
  publisher = {Springer-Verlag, New York},
  year      = {2002}
}

@article{KMR14,
    author =     "Z. Klagsbrun and B. Mazur and K. Rubin",
    title =      "A {Markov model for Selmer ranks} in families of twists",
    journal =    "Compositio Mathematica",
    volume =     "150",
    pages =      "1077--1106",
    year =       "2014"
}

@article{MR07,
    author = "B. Mazur and K. Rubin",
    title = "Finding large {Selmer} rank via an arithmetic theory of local constants",
    journal = "Annals of Mathematics",
    volume = "166",
    pages = "579--612",
    year = "2007"
}

@article{MRS07,
    author = "B. Mazur and K. Rubin and A. Silverberg",
    title = "Twisting commutative algebraic groups",
    journal = "Journal of Algebra",
    volume = "314",
    number = "1",
    pages = "419--438",
    year = "2007"
}

@article{PR12,
    author =     "B. Poonen and E. Rains",
    title =      "Random maximal isotropic subspaces and {Selmer} groups",
    journal =    "Journal of the American Mathematical Society",
    volume =     "25",
    number =     "1",
    pages =      "245--269",
    year =       "2012"
}

@article{KP27,
    author = "P. Koymans and C. Pagano",
    title = "Hilbert's tenth problem via additive combinatorics",
    journal = "To appear at Journal of the American Mathematical Society",
    volume = "40",
    pages = "195--234",
    year = "2027"
}

@misc{KPS24,
    author = "P. Koymans and C. Pagano and E. Sofos",
    title = "Elliptic fibrations and $3 \cdot 2^k$",
    year = "2024",
    eprint = "2409.02080",
    archivePrefix = "arXiv",
    primaryClass = "math.NT",
    note = "arXiv:2409.02080"
}

@misc{KS26,
    author = "P. Koymans and A. Smith",
    title = "Tamagawa ratios and unbounded Selmer moments",
    year = "2026",
    eprint = "2606.31649",
    archivePrefix = "arXiv",
    primaryClass = "math.NT",
    note = "arXiv:2606.31649"
}

@misc{BM25,
    author = "A. Bartel and A. Morgan",
    title = "Galois module structures and the {Hasse Principle} in twist families via the distribution of {Selmer groups}",
    year = "2025",
    eprint = "2508.14026",
    archivePrefix = "arXiv",
    primaryClass = "math.NT",
    note = "arXiv:2508.14026"
}

@misc{EL23,
    title = {Homological stability for generalized {Hurwitz} spaces and {Selmer} groups in quadratic twist families over function fields},
    author = {J. Ellenberg and A. Landesman},
    year = {2023},
    eprint = {2310.16286},
    archivePrefix={arXiv},
    primaryClass={math.NT},
    note = {arXiv:2310.16286}
}

@misc{LL25,
    title = {The stable homology of {Hurwitz} modules and applications},
    author = {A. Landesman and I. Levy},
    year = {2025},
    eprint = {2510.02068},
    archivePrefix={arXiv},
    primaryClass={math.NT},
    note = {arXiv:2510.02068}
}

@article{Sm22_01,
      title={The distribution of $\ell^{\infty}$-{S}elmer groups in degree $\ell$ twist families {I}}, 
      author={A. Smith},
      year={2026},
      journal = {Journal of the American Mathematical Society},
      volume = {39},
      number = {1},
      pages = {1--72}
}

@article{Sm22_02,
      title={The distribution of $\ell^{\infty}$-{S}elmer groups in degree $\ell$ twist families {II}}, 
      author={A. Smith},
      year={2026},
      journal = {Journal of the American Mathematical Society},
      volume = {39},
      number = {2},
      pages = {453--514}
}
